\documentclass[12pt]{amsart}

\usepackage{amsmath}
\usepackage{amsfonts}
\usepackage{amssymb}
\usepackage{graphicx}
\usepackage{hyperref}
\usepackage{mathrsfs}
\usepackage{pb-diagram}
\usepackage{epstopdf}
\usepackage{comment}

\usepackage{amscd,amsthm,curves,enumerate,latexsym,multibox}
\usepackage{bbm}
\usepackage{multirow}
\usepackage{mathrsfs}
\usepackage[mathscr]{eucal}
\usepackage{epsfig,epsf,epic}
\usepackage{epstopdf}

\newtheorem{thm}{Theorem}[section]
\newtheorem{lem}[thm]{Lemma}
\newtheorem{cor}[thm]{Corollary}
\newtheorem{prop}[thm]{Proposition}
\newtheorem{defn}[thm]{Definition}
\newtheorem{ex}[thm]{Example}
\newtheorem{rmk}[thm]{Remark}
\newtheorem{cnjc}[thm]{Conjecture}

\numberwithin{equation}{section}

\newcommand{\dd}[1]{\frac{\partial}{\partial #1}}

\newcommand{\pairing}[2]{\left( #1\, , \, #2 \right)}

\newcommand{\integer}{\mathbb{Z}}
\newcommand{\rat}{\mathbb{Q}}
\newcommand{\real}{\mathbb{R}}
\newcommand{\cpx}{\mathbb{C}}

\newcommand{\C}{\mathbb{C}}
\newcommand{\Z}{\mathbb{Z}}
\newcommand{\R}{\mathbb{R}}

\newcommand{\consti}{\mathbf{i}\,}
\newcommand{\conste}{\mathbf{e}}

\newcommand{\proj}{\mathbb{P}}
\newcommand{\tor}{\mathbf{T}}
\newcommand{\pt}{\mathrm{pt}}
\newcommand{\eff}{\mathrm{eff}}
\newcommand{\bP}{\mathbb{P}}
\newcommand{\cD}{\mathscr{D}}

\newcommand{\moduli}{\mathcal{M}}
\newcommand{\cM}{\mathcal{M}}

\newcommand{\virt}{\mathrm{virt}}
\newcommand{\reg}{\mathrm{reg}}
\newcommand{\cF}{\mathcal{F}}

\newcommand{\Aut}{\mathrm{Aut}}

\newcommand{\codim}{\mathrm{codim}}
\newcommand{\Dom}{\mathrm{Dom}}
\newcommand{\Image}{\mathrm{Im}}

\newcommand{\QH}{\mathrm{QH}}
\newcommand{\Jac}{\mathrm{Jac}}
\newcommand{\Kcone}{\mathcal{K}}
\newcommand{\nef}{\mathrm{nef}}
\newcommand{\Sym}{\mathrm{Sym}}
\newcommand{\uf}{\underline{f}}

\begin{document}

\title[Open GW, mirror maps, and Seidel representations for toric]{Open Gromov-Witten invariants, mirror maps, and Seidel representations for toric manifolds}

\author[Chan]{Kwokwai Chan}
\address{Department of Mathematics\\ The Chinese University of Hong Kong\\ Shatin \\ Hong Kong}
\email{kwchan@math.cuhk.edu.hk}
\author[Lau]{Siu-Cheong Lau}
\address{Department of Mathematics\\Harvard University\\ One Oxford Street, Cambridge\\ MA 02138\\ USA}
\email{s.lau@math.harvard.edu}
\author[Leung]{Naichung Conan Leung}
\address{The Institute of Mathematical Sciences and Department of Mathematics\\ The Chinese University of Hong Kong\\ Shatin \\ Hong Kong}
\email{leung@math.cuhk.edu.hk}
\author[Tseng]{Hsian-Hua Tseng}
\address{Department of Mathematics\\ Ohio State University\\ 100 Math Tower, 231 West 18th Ave. \\ Columbus \\ OH 43210\\ USA}
\email{hhtseng@math.ohio-state.edu}

\date{\today}

\begin{abstract}
Let $X$ be a compact toric K\"ahler manifold with $-K_X$ nef. Let $L\subset X$ be a regular fiber of the moment map of the Hamiltonian torus action on $X$. Fukaya-Oh-Ohta-Ono \cite{FOOO1} defined open Gromov-Witten (GW) invariants of $X$ as virtual counts of holomorphic discs with Lagrangian boundary condition $L$. We prove a formula which equates such open GW invariants with closed GW invariants of certain $X$-bundles over $\mathbb{P}^1$ used to construct the Seidel representations \cite{seidel97, McDuff_seidel} for $X$. We apply this formula and degeneration techniques to explicitly calculate all these open GW invariants. This yields a formula for the disc potential of $X$, an enumerative meaning of mirror maps, and a description of the inverse of the ring isomorphism of Fukaya-Oh-Ohta-Ono \cite{FOOO10b}.
\end{abstract}
\maketitle

\section{Introduction} \label{intro}

\subsection{Statements of results}
Let $X$ be a complex $n$-dimensional compact toric manifold equipped with a toric K\"ahler form $\omega$. $X$ admits a Hamiltonian action by a complex torus $\tor_\mathbb{C}\simeq (\mathbb{C}^*)^n$. Let $L\subset X$ be a regular fiber of the associated moment map. We will call $L$ a {\em Lagrangian torus fiber} because it is a Lagrangian submanifold of $X$ diffeomorphic to $(S^1)^n$. Let $\beta\in \pi_2(X, L)$ be a relative homotopy class represented by a disc bounded by $L$.  In \cite{FOOO1}, Fukaya-Oh-Ohta-Ono defined the {\em genus 0 open Gromov-Witten (GW) invariant} $n_1(\beta)\in \mathbb{Q}$ as a virtual count of holomorphic discs in $X$ bounded by $L$ representing the class $\beta$; the precise definition of $n_1(\beta)$ is reviewed in Definition \ref{open GW}. These invariants assemble to a generating function $W^{\mathrm{LF}}$ called the {\em disc potential} of $X$ (see Definition \ref{defn disc potential}).

The disc potential $W^{\mathrm{LF}}$ plays a fundamental role in the Lagrangian Floer theory of $X$ (hence the superscript ``LF'').  It was used by Fukaya-Oh-Ohta-Ono \cite{FOOO1,FOOO2,FOOO10b} to detect non-displaceable Lagrangian torus fibers in $X$.  Indeed, the $A_\infty$-algebra, which encodes all symplectic information of a Lagrangian torus fiber, is determined by $W^{\mathrm{LF}}$ and its derivatives. Furthermore, in an upcoming work, Abouzaid-Fukaya-Oh-Ohta-Ono show that the Fukaya category of $X$ is generated by Lagrangian torus fibers.  So $W^{\mathrm{LF}}$ completely determines the Fukaya category of $X$. On the other hand, the potential $W^{\mathrm{LF}}$ is also very important in the study of mirror symmetry because it serves as the Landau-Ginzburg mirror of $X$ and its Jacobian ring determines the quantum cohomology of $X$ \cite{FOOO10b}.

Open GW invariants are in general very difficult to compute because the obstruction of the corresponding moduli space can be highly non-trivial.  For Fano toric manifolds where the obstruction bundle is trivial, open GW invariants were computed by Cho-Oh \cite{cho06}.  The next simplest non-trivial example, which is the Hirzebruch surface $\mathbb{F}_2$, was computed by Auroux \cite{auroux09} using wall-crossing techniques\footnote{Auroux \cite{auroux09} also computed the open GW invariants for the Hirzebruch surface $\mathbb{F}_3$ using the same method.} and by Fukaya-Oh-Ohta-Ono \cite{FOOO10} using degenerations.  Later, under certain strong restrictions on the geometry of the toric manifolds, open GW invariants were computed in \cite{Chan10, chan-lau, CLL, CLT11, LLW10, LLW_surfaces}.

One main purpose of this paper is to compute the open GW invariants $n_1(\beta)$ for {\em all compact semi-Fano toric manifolds}. By definition, a toric manifold $X$ is {\em semi-Fano} if $-K_X$ is nef, i.e. $-K_X\cdot C\geq 0$ for every holomorphic curve $C\subset X$. Let $\beta\in \pi_2(X,L)$ be a disc class\footnote{By dimension reasons only classes $\beta$ of Maslov index 2 can have non-zero $n_1(\beta)$. See Section \ref{Sect open_GW_toric} for details.} of Maslov index 2 such that $n_1(\beta)\neq0$. By the results of Cho-Oh \cite{cho06} and Fukaya-Oh-Ohta-Ono \cite{FOOO1} (see also Lemma \ref{Lem Maslov-two stable disc}), the class $\beta$ must be of the form $\beta=\beta_l+\alpha$, where $\beta_l\in \pi_2(X, L)$ is the {\em basic disc class} associated to a toric prime divisor $D_l$ (the class of the unique Maslov index 2 embedded disk intersecting $D_l$ at a point; see \cite[Definition 7.1]{cho06}) and $\alpha\in H_2^\text{eff}(X) \subset H_2(X, \mathbb{Z})$ is an effective curve class with {\em Chern number} $c_1(\alpha):=-K_X\cdot \alpha=0$. Define the following generating function (see Definition \ref{defn disc potential} for more details):
$$\delta_l(q):=\sum_{\substack{\alpha\in H_2^\text{eff}(X)\setminus \{0\} \\ c_1(\alpha)=0}} n_1(\beta_l+\alpha)q^\alpha.$$

One of our main results is an explicit formula for the generating function $\delta_l(q)$ which we now explain. The {\em toric mirror theorem} of Givental \cite{givental98} and Lian-Liu-Yau \cite{LLY3}, as recalled in Theorem \ref{thm toric mir thm}, states that there is an equality $$I(\check{q},z)=J(q(\check{q}),z),$$
where $I(\check{q},z)$ is the combinatorially defined {\em $I$-function} of $X$ (see Definition \ref{defn I}), $J(q,z)$ is a certain generating function of closed GW invariants of $X$ called the {\em $J$-function} (see Equation \eqref{J-function}), and $q(\check{q})$ is the {\em mirror map} in Definition \ref{Defn mir map}. Our formula for $\delta_l(q)$ reads as follows:

\begin{thm} \label{thm delta_intro}
Let $X$ be a compact semi-Fano toric manifold. Then
$$1+\delta_l(q)= \exp(g_l(\check{q}(q))),$$ where
\begin{equation}\label{eqn:funcn_g}
g_l(\check{q}):=\sum_{d}\frac{(-1)^{(D_l\cdot d)}(-(D_l\cdot d)-1)!}{\prod_{p\neq l} (D_p\cdot d)!}\check{q}^d
\end{equation}
where the summation is over all effective curve classes $d\in H_2^\text{eff}(X)$ satisfying
$$-K_X\cdot d=0, D_l\cdot d<0 \text{ and } D_p\cdot d \geq 0 \text{ for all } p\neq l$$
and $\check{q}=\check{q}(q)$ is the inverse of the mirror map $q=q(\check{q})$.
\end{thm}

The mirror map $q=q(\check{q})$ is combinatorially defined, and its inverse $\check{q}=\check{q}(q)$ can be explicitly computed, at least recursively. So our formula provides an effective calculation for all genus 0 open GW invariants. It may also be inverted to give a formula which expresses the inverse mirror map $\check{q}(q)$ in terms of genus 0 open GW invariants (see Corollary \ref{cor inv mir map}), thereby giving the inverse mirror map an {\em enumerative meaning} in terms of disc counting.

Our calculation of open GW invariants can also be neatly stated in terms of the disc potential, giving the following {\em open mirror theorem}:
\begin{thm} \label{thm W equal}
Let $X$ be a compact semi-Fano toric manifold. Then
\begin{equation} \label{eqn same W}
W^{\mathrm{LF}}_{q}=\tilde{W}^{\mathrm{HV}}_{\check{q}(q)},
\end{equation}
where $\tilde{W}^{\mathrm{HV}}$ is the {\em Hori-Vafa superpotential} for $X$ in a certain explicit choice of coordinates of $(\cpx^*)^n$; see Definition \ref{defn HV} and Equation \eqref{eqn W tilde}.
\end{thm}

While the disc potential $W^{\mathrm{LF}}$ is a relatively new object invented to describe the symplectic geometry of $X$, the Hori-Vafa superpotential $\tilde{W}^{\mathrm{HV}}$ has been studied extensively in the literature. Thus existing knowledge on $\tilde{W}^{\mathrm{HV}}$ can be employed to understand the disc potential $W^{\mathrm{LF}}$ better via Theorem \ref{thm W equal}.

In particular, since $\tilde{W}^{\mathrm{HV}}$ is written in terms of (inverse) mirror maps which are known to be convergent, it follows that the coefficients of the disc potential $W^{\mathrm{LF}}$ are convergent power series as well (See Theorem \ref{thm conv}).

Furthermore, the mirror theorem \cite{givental98, LLY3} induces an isomorphism
\begin{equation}
\QH^*(X,\omega_q)\overset{\simeq}{\longrightarrow} \Jac(\tilde{W}^{\mathrm{HV}}_{\check{q}})
\end{equation}
between the quantum cohomology of $X$ and the Jacobian ring of the Hori-Vafa superpotential when $X$ is semi-Fano.  Combining with Theorem \ref{thm W equal}, this gives another proof of the following
\begin{cor}[FOOO's isomorphism \cite{FOOO10b} for small quantum cohomology in semi-Fano case\footnote{Fukaya-Oh-Ohta-Ono \cite{FOOO10b} proved a ring isomorphism between the {\em big} quantum cohomology ring of {\em any} compact toric manifold $X$ and the Jacobian ring of its {\em bulk-deformed} potential function; our results give such an isomorphism for the {\em small} quantum cohomology of a semi-Fano toric manifold $X$.}]
Let $X$ be a compact semi-Fano toric manifold.  Then there exists an isomorphism
\begin{equation}\label{FOOO_isom_intro}
\QH^*(X,\omega_q)\overset{\simeq}{\longrightarrow} \Jac(W_q^{\mathrm{LF}})
\end{equation}
between the small quantum cohomology ring of $X$ and the Jacobian ring of  $W_q^{\mathrm{LF}}$.
\end{cor}

On the other hand, McDuff-Tolman \cite{McDuff-Tolman} constructed a presentation of $\QH^*(X,\omega_q)$ using Seidel representations (\cite{seidel97, McDuff_seidel}) and showed that it is abstractly isomorphic to the Batyrev presentation \cite{batyrev93}. This was exploited by Fukaya-Oh-Ohta-Ono \cite{FOOO10b} in their proof of the injectivity of the homomorphism \eqref{FOOO_isom_intro} but they did not specify the precise relations between \eqref{FOOO_isom_intro} and Seidel elements. Using our results on open GW invariants, we deduce that:
\begin{thm} \label{thm Seidel to gen intro}
Suppose $X$ is semi-Fano. Then the isomorphism \eqref{FOOO_isom_intro} maps the (normalized) Seidel elements $S^\circ_l\in \QH^*(X, \omega_q)$ (see Section \ref{Sect Seidel_rep}) to the generators $Z_1, \ldots, Z_m$ of the Jacobian ring $\Jac(W_q^{\mathrm{LF}})$, where $Z_l$ are monomials defined by Equation \eqref{Eqn Z}.
\end{thm}

We conjecture that Theorem \ref{thm Seidel to gen intro}, which provides a highly non-trivial relation between open Gromov-Witten invariants and Seidel representations, holds true for {\em all} toric manifolds; see Conjecture \ref{conj:seidel_elt} for the precise statement.

\subsection{Outline of methods}
The {\em closed} GW theory for toric manifolds has been studied extensively, and various powerful computational tools such as virtual localization are available. The situation is drastically different for {\em open} GW theory with respect to Lagrangian torus fibers -- the open GW invariants, which are defined using moduli spaces of stable discs that could have very sophisticated structures, are very hard to compute in general, especially because of the lack of localization techniques.\footnote{This is in sharp contrast with the situation for Aganagic-Vafa type Lagrangian submanifolds in toric Calabi-Yau 3-folds, where the open GW invariants are practically defined by localization formulas and can certainly be evaluated using them.}


In this paper we study the problem of computing open GW invariants via a geometric approach which we outline as follows. As we mention above, for $\beta\in \pi_2(X, L)$ of Maslov index 2 with $n_1(\beta) \neq 0$, a stable disc representing $\beta$ must have its domain being the union of a disc $D$ and a collection of rational curves. Na\"ively one may hope to ``cap off'' the disc $D$ by finding another disc $D'$ and gluing $D$ and $D'$ together along their boundaries to form a sphere. If this can be done, it is then natural to speculate that the open GW invariants we want to compute are equal to certain closed GW invariants. This idea was first worked out in \cite{Chan10} for toric manifolds of the form $X=\mathbb{P}(K_Y\oplus \mathcal{O}_Y)$ where $Y$ is a compact Fano toric manifold; in that case the $\mathbb{P}^1$-bundle structure on $X$ provides a way to find the needed disc $D'$. The same idea was applied in subsequent works \cite{LLW10, LLW_surfaces, CLL, chan-lau, CLT11}, and it gradually became clear that in more general situations, we need to work with a target space different from $X$ in order to find the ``capping-off'' disc $D'$.

One novelty of this paper is the discovery that {\em Seidel spaces} are the correct spaces to use in the case of semi-Fano toric manifolds. Given a toric manifold $X$, let $D_l$ be a toric prime divisor and let $v_l$ be the primitive generator of the corresponding ray in the fan. Then $-v_l$ defines a $\mathbb{C}^*$-action on $X$. Let $\mathbb{C}^*$ act on $\mathbb{C}^2\setminus \{0\}$ by $z\cdot (u,v):=(zu, zv), z\in \mathbb{C}^*, (u,v)\in \mathbb{C}^2\setminus \{0\}$. The {\em Seidel space} associated to the $\mathbb{C}^*$-action defined by $-v_l$ is the quotient $$E_l^{-}:=(X\times (\mathbb{C}^2\setminus \{0\}))/\mathbb{C}^*.$$
By construction, $E_l^-$ is also a toric manifold, and there is a natural map $E_l^-\to \mathbb{P}^1$ giving $E_l^-$ the structure of a fiber bundle over $\mathbb{P}^1$ with fiber $X$. The toric data of $E_l^-$, as well as geometric information such as its Mori cone, can be explicitly described; see Section \ref{Sect Seidel_rep}.

Recall that the disc classes which give non-zero open GW invariants are of the form $\beta_l+\alpha$, where $\beta_l$ is the basic disc class associated to the toric prime divisor $D_l$ for some $l$, and $\alpha\in H_2^\text{eff}(X)$ is an effective curve class with $c_1(\alpha)=0$.  We prove the following
\begin{thm}[See Theorem \ref{Thm open-closed}]\label{thm:open-closed_intro}
Let $X$ be a compact semi-Fano toric manifold defined by a fan $\Sigma$, and $L\subset X$ a Lagrangian torus fiber. Let $P$ be the fan polytope of $X$, which is the convex hull of minimal generators of rays in $\Sigma$.
Then for minimal generators $v_l,v_k$ of rays in $\Sigma$, we have
\begin{equation}\label{open_closed_intro}
n^X_{1,1}(\beta_l+\alpha; D_k, [\mathrm{pt}]_{L}) = \langle D_k^{{E_l^-}}, [\mathrm{pt}]_{E_l^-} \rangle^{E_l^-,\, \sigma_l^- \, \reg}_{0,2,\sigma_l^-+\alpha},
\end{equation}
when $v_k \in F(v_l)$ and $\alpha \in H_2^{\eff,c_1=0}(X)$ satisfies $D_i \cdot \alpha = 0$ whenever $v_i \not\in F(v_l)$, where $F(v_l)$ is the minimal face of $P$ containing $v_l$; and $n^X_{1,1}(\beta_l+\alpha; D_k, [\mathrm{pt}]_{L}) = 0$ otherwise.
\end{thm}

The left-hand side of \eqref{open_closed_intro} is the open GW invariant defined in Definition \ref{open GW}, which roughly speaking counts discs of classes $\beta_l+\alpha$ meeting a fixed point in $L$ at the boundary marked point and meeting the divisor $D_k$ at the interior marked point. The invariant $n^X_{1,1}(\beta_l+\alpha; D, [\mathrm{pt}]_{L})$ is related to the previous invariant $n_1(\beta_l+\alpha)$ via the {\em divisor equation} proved by Fukaya-Oh-Ohta-Ono \cite[Lemma 9.2]{FOOO2} (see also Theorem \ref{thm div eq open}):
$$n^X_{1,1}(\beta_l+\alpha; D, [\mathrm{pt}]_{L})= \big( D\cdot(\beta_l+\alpha) \big) n_1(\beta_l+\alpha).$$
On the right-hand side of \eqref{open_closed_intro} we have the two-point {\em closed} $\sigma_l^-$-regular GW invariant
\begin{equation}\label{closed_inv_intro}
\langle D_k^{{E_l^-}}, [\mathrm{pt}]_{{E_l^-}} \rangle^{{E_l^-, \, \sigma_l^- \, \reg}}_{0,2,\sigma_l^-+\alpha}
\end{equation}
of the Seidel space $E_l^-$, which is the integration over a connected component of the moduli $\moduli^{E^-}_{0,2,\sigma^-+\alpha}(D^{{E^-}}, \pt)$ where $\sigma_l^-$ is the zero section class of the Seidel space $E_l^-$; see Section \ref{Sect Seidel_rep} for the notations.

The geometric idea behind the proof of \eqref{open_closed_intro} is the following.  If $v_k \not\in F(v_l)$, then $D_k \cdot (\beta_l + \alpha) = 0$ (Proposition \ref{prop:Cin}), and so $n^X_{1,1}(\beta_l+\alpha; D_k, [\mathrm{pt}]_{L}) = 0$.  Now consider the more difficult case $v_k \in F(v_l)$.  A stable disc representing the class $\beta_l+\alpha$ is a union of a disc $\Delta$ in $X$ representing $\beta_l$ and a rational curve $C$ in $X$ representing $\alpha$.  We identify $X$ with the fiber of $E_l^-\to \mathbb{P}^1$ over $0\in \mathbb{P}^1$ and consider $C$ as in $E_l^-$. The key point is that the disc $\Delta$ in $X$ bounded by a Lagrangian torus fiber $L$ of $X$ can be identified with a disc $\widetilde{\Delta}$ in $E_l^-$ bounded by a Lagrangian torus fiber $\widetilde{L}$ of $E_l^-$, and there exists a ``capping-off'' disc $\widetilde{\Delta}'$ in $E_l^-$ which can be glued together with $\widetilde{\Delta}$ to form a rational curve representing a section $\sigma_l^-$ of $E_l^-\to \mathbb{P}^1$. This idea, which is illustrated in Figure 1, allows us to identify the relevant moduli spaces. A further analysis on their Kuranishi structures yields the formula \eqref{open_closed_intro}.

\begin{rmk}
In this paper we consider open GW invariants defined using Kuranishi structures. However we would like to point out that the formula \eqref{open_closed_intro} in Theorem \ref{thm:open-closed_intro} remains valid whenever reasonable structures are put on the moduli spaces to define GW invariants. This is because our ``capping off'' argument is geometric in nature and it identifies the deformation and obstruction theories of the two moduli problems on the nose.
\end{rmk}

\begin{figure}[htp]
\begin{center}
\includegraphics[scale=0.7]{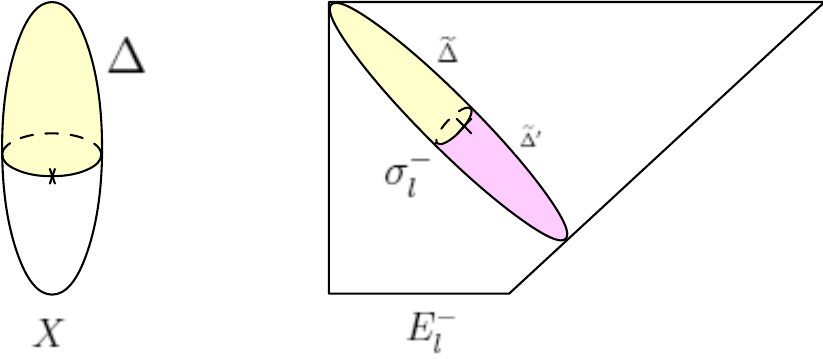}
\end{center}
\caption{Relating disc invariants to GW invariants of the Seidel space.  A disc $\Delta$ in $\proj^1$ bounded by a torus fiber is `pushed' into the associated Seidel space and compactified to a sphere in the section class $\sigma_l^-$.}
\label{fig disc to Seidel}
\end{figure}

Our formula \eqref{open_closed_intro} reduces the computation of open GW invariants to the computation of the closed GW invariants \eqref{closed_inv_intro}. But since $E_l^-$ is {\em not} semi-Fano, and these invariants are some more refined closed GW invariants of $E_l^-$ (see Definition \ref{def:regGW}), computing \eqref{closed_inv_intro} presents a non-trivial challenge. 

Our calculation of \eqref{closed_inv_intro} uses several techniques. First of all, the Seidel space $E_l$ is semi-Fano. Gonz\'alez-Iritani \cite{G-I11} calculated the corresponding Seidel element $S_l$ using the $J$-function of $E_l$ and applying the toric mirror theorem, and expressed it in terms of the so-called {\em Batyrev elements} $B_l$ (see Proposition \ref{prop Seidel Batyrev}). We then write the divisor $D_k$ in terms of the Batyrev elements $B_l$ (see Proposition \ref{Prop D-B}). Finally, a degeneration technique for closed GW invariants, which was used to derive the composition law for Seidel representations, can be exploited to analyze the invariants $\langle D_k^{{E_l^-}}, [\mathrm{pt}]_{{E_l^-}} \rangle^{{E_l^-, \, \sigma_l^- \, \reg}}_{0,2,\sigma_l^-+\alpha}$ and deduce Theorem \ref{thm delta_intro}. The details are given in Section \ref{sect compute Seidel}.
\begin{rmk}
\hfill
\begin{enumerate}


\item
Equation \eqref{open_closed_intro} is a relation between open and closed GW invariants. An ``open/closed relation'' of somewhat different flavor is present in open GW theory of toric Calabi-Yau 3-folds with respect to Aganagic-Vafa type Lagrangian branes; see \cite{Mayr02, Lerche_Mayr, Mayr02-2, Lerche_Mayr_Warner}.

\item
During the preparation of this paper, we learnt of an independent work of Gonz\'alez-Iritani \cite{G-I12} in which an alternative approach to Theorem \ref{thm W equal} based on a conjectural degeneration formula for open GW invariants is developed.
\end{enumerate}
\end{rmk}

The rest of this paper is organized as follows. Section \ref{Sect open_GW_toric} contains a brief review of open GW invariants of toric manifolds. In Section \ref{sect tor mir} we review the toric mirror theorem \cite{givental98, LLY3}, Hori-Vafa superpotentials, mirror maps, and related materials. In Section \ref{Sect Seidel_rep} we recall some basic materials on Seidel representations of toric manifolds. In Section \ref{Sect open_closed} we prove the relation \eqref{open_closed_intro} between open and closed GW invariants. In Section \ref{sect compute Seidel} we calculate the closed GW invariants which appear in \eqref{open_closed_intro} and prove our main Theorems \ref{thm delta_intro}, \ref{thm W equal}, \ref{thm conv}, \ref{thm Seidel to gen intro}.

\section*{Acknowledgement}
We express our deep gratitude to Eduardo Gonz\'alez and Hiroshi Iritani for various illuminating discussions on this subject and informing us another insightful approach to this problem. We thank Kaoru Ono and Cheol-Hyun Cho for interesting and useful discussions. We are deeply indebted to several referees for many useful comments and suggestions and for pointing out essential issues, which help to greatly improve this paper. Parts of this work were carried out when the authors met at The Chinese University of Hong Kong, University of Wisconsin-Madison, Ohio State University, and Kavli IPMU. We thank these institutions for hospitality and support.

The work of K. C. was partially supported by a grant from the Research Grants Council of the Hong Kong Special Administrative Region, China (Project No. CUHK404412). The work of S.-C. L. was supported by Harvard University. The work of N. C. L. described in this paper was substantially supported by a grant from the Research Grants Council of the Hong Kong Special Administrative Region, China (Project No. CUHK401809). H.-H. T. was supported in part by a Simons Foundation Collaboration Grant.

\begin{center}
{\Large \bf A list of notations}
\end{center}
\begin{tabular}{p{1.75cm}p{14cm}}
$N$ & A lattice $\Z^n$ \\
$\tor$ & The real torus $\R^n / \mathbb{Z}^n$ and a Lagrangian torus fiber\\
$X$ & Toric manifold of dimension $n$ \\
$\Sigma$ & The defining fan of a toric manifold \\
$m$ & The number of primitive generators in $\Sigma$ \\
$K_X$ & The canonical line bundle of $X$ \\
$\Kcone_X$ & The K\"ahler cone of $X$ \\
$\Kcone_X^\cpx$ & The complexified K\"ahler cone of $X$ \\
$D_j$ & Toric prime divisor (as a cycle) \\
$[D_j]$ & The class of $D_j$ in $H^2(X)$ \\
$v_j$ & Primitive generator of a ray of a fan \\
$\nu_j$ & The dual basis of $\{v_1,\ldots,v_n\}$ \\
$\Psi_k$ & The dual basis of $[D_{n+1}], \ldots, [D_m]$ in $H^2(X)$ \\
$\Phi_a$ & Homogeneous basis of $H^*(X)$ \\
$\Phi^a$ & The dual basis of $\Phi_a$ with respect to Poincar\'e pairing \\
$L$ & A regular moment-map fiber of a toric manifold \\
$\beta$ & A disc class bounded by a regular moment-map fiber of a toric manifold \\
$H_2^{\eff}$ & The set of effective curve classes in $H_2$ \\
$\beta_j$ & Basic disc class of a toric manifold \\
$\cM_{l,k}(\beta)$ & Compactified moduli of stable discs in $\beta$ with $l$ interior and $k$ boundary marked points \\
$\cM_{k}(\beta)$ & Compactified moduli of stable discs in $\beta$ with $k$ boundary marked points \\
$\cM_{1,1}(\beta;C)$ & The fiber product $\cM_{1,1}(\beta) \times_X C$ \\
$n_1(\beta)$ & One-pointed open Gromov-Witten invariant of $\beta$ \\
$n_{1,1}$ & Open Gromov-Witten invariant with one interior and one boundary insertion \\
$W^{\mathrm{LF}}$ & Disc potential of a toric manifold \\
$\tilde{W}^{\mathrm{HV}}$ & Hori-Vafa superpotential of a toric manifold\\
$q_i$ ($\check{q}_i$) & K\"ahler parameter (resp. mirror complex parameter) (for $i=1,\ldots,m-n$) \\
$Q_k$ & Extended K\"ahler parameter (for $k=1,\ldots,m$) \\
$z$ & $\mathbb{S}^1$-equivariant parameter \\
$z_j$ & Mirror complex coordinates (for $j=1,\ldots,n$) \\
$Z_i$ & $z_i$ if $i=1,\ldots,n$ and $q_{l-n}\prod_{i=1}^n z_i^{(\nu_i, v_l)}$ if $i=n+1,\ldots,m$\\
$B_{n+i}$ & Batyrev element (for $i=1,\ldots,m-n$) \\
$\tilde{B}_k$ & Extended Batyrev element (for $k=1,\ldots,m$) \\
$q(\check{q})$ & Mirror map from mirror complex moduli to K\"ahler moduli \\
$\check{q}(q)$ & Inverse mirror map from K\"ahler moduli to mirror complex moduli \\
$Q(\check{Q})$ & Extended mirror map \\
$\check{Q}(Q)$ & Inverse extended mirror map \\
$I(\check{q},z)$ & Givental $I$-function defined from combinatorial data of a toric manifold \\
$J(q,z)$ & Givental $J$-function defined from Gromov-Witten invariants \\
$\QH$ & Quantum cohomology \\
$\Jac$ & Jacobian ring \\
$E_j$ ($E_j^-$) & Seidel space associated to a primitive generator $v_j$ (resp. $-v_j$) \\
$S_j$ ($S_j^-$) & Seidel element associated to a primitive generator $v_j$ (resp. $-v_j$)\\
$\sigma_j$ ($\sigma_j^-$) & Zero section class of $E_j$ (resp. $E_j^-$) \\
$\sigma_\infty$ ($\sigma_\infty^-$) & Infinite section class of $E_j$ (resp. $E_j^-$) \\
\end{tabular}

\section{A brief review on open GW invariants of toric manifolds} \label{Sect open_GW_toric}

This section gives a quick review on toric manifolds and their open GW invariants which are the central objects to be studied in this paper.  For a nice exposition of toric varieties, the readers are referred to Fulton's book \cite{Fu}.  The Lagrangian Floer theory we use in this paper is developed by Fukaya-Oh-Ohta-Ono \cite{FOOO_I, FOOO_II, FOOO1, FOOO2, FOOO10b}.

We work with a projective toric $n$-fold $X$ equipped with a toric K\"ahler form.  Let $N\cong\integer^n$ be a lattice and let $\Sigma \subset N \otimes_\integer\real$ be the complete simplicial fan defining $X$.  The minimal generators of rays in $\Sigma$ are denoted by $v_j\in N$ for $j = 1, \ldots, m$.  Each $v_j$ corresponds to a toric prime divisor denoted by $D_j \subset X$.  There is an action on $X$ by the torus $\tor_\cpx = (N \otimes_\integer \cpx) / N$ which preserves the K\"ahler structure, and the associated moment map which maps $X$ to a polytope in $(N \otimes_\integer\real)^*$.  Each regular fiber of the moment map is a Lagrangian submanifold and it is a free orbit under the real torus $\tor = (N \otimes_\integer \real) / N$ action.  By abuse of notation we also denote such a fiber by $\tor$, and call it a Lagrangian torus fiber.  $X$ is said to be \emph{semi-Fano} if $-K_X$ is numerically effective, i.e. $-K_X \cdot C \geq 0$ for any holomorphic curve $C$.  We call $-K_X \cdot C$ the {\em Chern number} of $C$ and denote it by $c_1 (C)$.

Let $X$ be a semi-Fano toric manifold equipped with a toric K\"ahler form.  Our goal is to compute the open GW invariants of $X$, which are rational numbers associated to disc classes $\beta \in \pi_2(X,\tor)$ bounded by a Lagrangian torus fiber. To define open GW invariants, recall that in the toric case the Maslov index of a disc class $\beta$ is given by
$$\mu(\beta) = 2\sum_{j=1}^m D_j \cdot \beta,$$
where $D_j \cdot \beta$ is the intersection number of $\beta$ with the toric divisor $D_j$. By Cho-Oh \cite{cho06}, holomorphic disc classes in $\pi_2(X,\tor)$ are generated by basic disc classes $\beta_j$, $j=1,\ldots,m$, with $\partial \beta_j = v_j \in \pi_1(\tor)$.  Since $\mu(\beta_j) = 2$ for all $j=1,\ldots,m$, every non-constant holomorphic disc has Maslov index at least 2.

For each disc class $\beta \in \pi_2(X,\mathbf{T})$, Fukaya-Oh-Ohta-Ono \cite{FOOO1,FOOO2} defined the moduli space
$$\mathcal{M}_{l,k} (\beta) = \mathcal{M}^{\mathrm{op}}_{l,k} (\beta) $$
of stable discs with $l$ interior marked points and $k$ boundary marked points representing $\beta$, which is oriented and compact.  When $l=0$, we simply denote $\mathcal{M}^{\mathrm{op}}_{0,k} (\beta)$ by $\mathcal{M}^{\mathrm{op}}_{k} (\beta)$.  Here we use the superscript ``op'' to remind ourselves that it is the moduli space for defining open GW invariants.  Later we will use the superscript ``cl'' (which stands for ``closed'') for the moduli space of stable maps from rational curves.

The main difficulty in defining the invariants is the lack of transversality: the actual dimension of $\mathcal{M}_{l,k} (\beta)$ in general is higher than its expected (real) dimension $n+\mu(\beta)+k+2l-3$.  To tackle this problem, Fukaya-Oh-Ohta-Ono analyzed the obstruction theory and used the torus action on $\mathcal{M}_{l,k} (\beta)$ to construct a virtual fundamental chain $[\mathcal{M}_{l,k} (\beta)]_{\mathrm{virt}}$ which is intrinsic to the disc moduli.  By using the evaluation map $\mathcal{M}_{l,k} (\beta) \to X^l\times \tor^k$, we shall identify $[\mathcal{M}_{l,k} (\beta)]_{\mathrm{virt}}$ as a $\mathbb{Q}$-chain of dimension $n+\mu(\beta)+k+2l-3$ in $X^l\times \tor^k$. In this paper we shall only need the cases when $k=1$ and $l$ is either $0$ or $1$.  When $l=0, k=1$ and $\mu(\beta)=2$, as non-constant stable discs bounded by $\mathbf{T}$ have Maslov indices at least 2, the moduli space $\mathcal{M}_l(\beta) = \mathcal{M}_{0,1}(\beta)$ has no codimension one boundary and so $[\mathcal{M}_1(\beta)]_{\mathrm{virt}}$ is actually a cycle.  (For a nice discussion of this, we refer the reader to \cite[Section 3]{auroux07}).  For $l=1, k=1$, we will consider
$\mathcal{M}_{1,1} (\beta;C) := \mathcal{M}_{1,1} (\beta) \times_X C$,
where $C \subset \bigcup_{j=1}^m D_j$ is a proper toric cycle (i.e. an algebraic cycle preserved by the torus action).  Since the interior marked point is constrained to map to $C$, it can never approach the boundary of the disc. Hence the moduli space does not have codimension one boundary and $[\moduli_{1,1} (\beta; C)]_{\virt}$ is again a cycle.  By Poincar\'e duality we will identify both $[\mathcal{M}_1(\beta)]_{\mathrm{virt}}$ and $[\mathcal{M}_{1,1}(\beta;C)]_{\mathrm{virt}}$ as cohomology classes in $H^*(\tor)$.

\begin{defn} [Open GW invariants \cite{FOOO1, FOOO2}] \label{open GW}
Let $X$ be a compact semi-Fano toric manifold and $\tor$ a Lagrangian torus fiber of $X$. We denote by $\pairing{\cdot}{\cdot}$ the Poincar\'e pairing on $H^*(\tor)$. The {\em one-point open GW invariant associated to a disc class $\beta \in \pi_2(X,\tor)$} is defined to be
$$n_1(\beta;[\pt]_\tor) = n_1(\beta) := \pairing{[\mathcal{M}_1 (\beta)]_{\mathrm{virt}}}{[\pt]_\tor} \in \rat.$$
For a proper toric cycle $C \subset X$ (i.e. an algebraic cycle invariant under the torus action on $X$ contained in $\bigcup_{j=1}^m D_j$) of real codimension $\codim(C)$ and a disc class $\beta \in \pi_2(X,\tor)$, let
$$\moduli_{1,1} (\beta; C) := \moduli_{1,1} (\beta) \times_X C$$
where the fiber product over $X$ is defined by the evaluation map $\mathrm{ev}_+: \moduli_{1,1} (\beta) \to X$ at the interior marked point and the inclusion map $C \hookrightarrow X$.  The expected dimension of $\moduli_{1,1} (\beta; C)$ is $n + \mu(\beta) - \codim(C)$.  The {\em one-point open GW invariant of class $\beta$ relative to $C$} is defined to be
$$ n_{1,1}(\beta;C,[\pt]_\tor) := \pairing{[\moduli_{1,1} (\beta; C)]_{\virt}}{[\pt]_\tor} \in \rat$$
where the torus action on $\moduli_{1,1} (\beta; C)$ is used to construct the virtual class $[\moduli_{1,1} (\beta; C)]_{\virt} \in H^*(\tor)$.
\end{defn}

Intuitively $n_1(\beta)$ counts stable discs in the class $\beta$ passing through a generic boundary marked point in $\tor$, while $n_{1,1}(\beta;C,[\pt]_\tor)$ counts stable discs in the class $\beta$ hitting the cycle $C$ at an interior marked point and passing through a generic boundary marked point in $\tor$.  Notice that by dimension counting, $n_1(\beta) \neq 0$ (resp. $n_{1,1}(\beta;C,[\pt]_\tor) \neq 0$) only when $\mu(\beta) = 2$ (resp. $\mu(\beta) = \codim(C)$).  Thus when $C$ is a toric divisor, we only need to consider those $\beta$ with $\mu(\beta) = 2$. We have the following analog of divisor equation in the open case:

\begin{thm}[See \cite{FOOO2}, Lemma 9.2] \label{thm div eq open}
For a toric divisor $D\subset X$ and a disc class $\beta\in \pi_2(X,\tor)$ with $\mu(\beta) = 2$, we have $ n_{1,1}(\beta;D,[\pt]_\tor) = (D\cdot\beta)\, n_1(\beta)$,
where $D\cdot\beta$ denotes the intersection number between $D$ and $\beta$.
\end{thm}

We can now define the disc potential.
First we make the following choices.  By relabeling the generators $\{v_j\}_{j=1}^m$ of rays if necessary, we may assume that $v_1, \ldots, v_n$ span a cone in the fan $\Sigma$ so that $\{v_1, \ldots, v_n\}$ gives a $\mathbb{Z}$-basis of $N$. Denote the dual basis by $\{\nu_k\}_{k=1}^n \subset M:=N^*$. Moreover, take the basis $\{\Psi_k\}_{k=1}^{m-n}$ of $H_2(X)$ where
\begin{equation}\label{H_2_basis}
 \Psi_k := -\sum_{p=1}^n \pairing{\nu_p}{v_{n+k}} \beta_p + \beta_{n+k} \in H_2(X), \quad \text{for } k = 1, \ldots, m-n.
\end{equation}
(Recall that the basic disc classes $\{\beta_j\}_{j=1}^m$ form a basis of $H_2(X,\tor)$.)  Note that $\Psi_k\in H_2(X)$ because
$\partial \Psi_k = -\sum_{p=1}^n\pairing{\nu_p}{v_{n+k}}v_p + v_{n+k} = -v_{n+k}+v_{n+k} = 0$.
Since $D_{n+r} \cdot \Psi_k = \delta_{kr}$ for all $k,r=1,\ldots,m-n$, the dual basis of $\{\Psi_k\}_{k=1}^{m-n}$ is given by $[D_{n+1}],\ldots,[D_{m}] \in H^2(X)$.

The basis $\{\Psi_k\}_{k=1}^{m-n}$ defines flat coordinates on $H^2(X,\cpx)/2\pi\consti H^2(X,\integer)$ by sending $$[\eta] \in H^2(X,\cpx)/2\pi\consti H^2(X,\integer)$$ to $q_k([\eta]) = q^{\Psi_k}(\eta) := \exp \left(\int_{\Psi_k}\eta\right)$ for $k=1,\ldots,m-n$.  Let $\Kcone_X$ denote the K\"ahler cone which consists of all K\"ahler classes on $X$.  The {\em complexified K\"ahler cone}
$$\Kcone^{\cpx}_X := \Kcone_X \oplus (\consti H^2(X,\real)/2\pi\consti H^2(X,\integer))$$
is embedded as an open subset of $H^2(X,\cpx)/2\pi\consti H^2(X,\integer)$ by taking $\omega \in \Kcone^{\cpx}_X$ to $-\omega \in H^2(X,\cpx)/2\pi\consti H^2(X,\integer)$.  Then $(q_k)_{k=1}^{m-n}$ pull back to give a flat coordinate system on $\Kcone^{\cpx}_X$.

Note that $[D_{n+1}],\ldots,[D_{m}] \in H^2(X)$ may not be nef (meaning their Poincar\'e pairings algebraic curves may be negative).  A theoretically better choice would be a nef basis of $H^2(X)$.  Taking its Poincar\'e dual basis gives another set of flat coordinates which we denote as $q^{\nef}_k$ on $H^2(X,\cpx)/2\pi\consti H^2(X,\integer) \cong (\C^*)^{m-n}$.  Then the {\em large radius limit} is defined by $q^{\nef} = 0$.  Since in most situations we work in $(\C^*)^{m-n}$ (except when we talk about convergence at $q^\nef = 0$), we may use the above more explicit coordinate system $q$.

The disc potential is defined by summing up the one-point open GW invariants for all $\beta \in \pi_2(X,\tor)$ weighted by $q^\beta$.  By dimension reasons, only those $\beta$ with Maslov index 2 contribute.  Since $X$ is a semi-Fano toric manifold, stable discs with Maslov index 2 must be of the form $\beta_j + \alpha$ for some basic disc class $\beta_j$ and $\alpha \in H_2^{\eff}(X)$ with $c_1(\alpha) = 0$. Here $H_2^{\eff}(X)\subset H_2(X,\mathbb{Z})$ is the semi-group of effective curve classes of $X$. More precisely,
\begin{lem}[\cite{cho06,FOOO1}] \label{Lem Maslov-two stable disc}
Let $X$ be a semi-Fano toric manifold and $\tor$ a Lagrangian torus fiber of $X$.  A stable disc in $\moduli^{\mathrm{op}}_{l,k}(\beta)$ for $k=1$ and $l=0,1$ where $\beta \in \pi_2(X,\tor)$ with $\mu(\beta) = 2$ is a union of a holomorphic disc component and a rational curve, which are attached to each other at only one nodal interior point.  The disc component represents the class $\beta_j$ for some $j=1, \ldots, m$, and the rational curve has $c_1 = 0$.  Thus the class of every stable disc is of the form $\beta_j + \alpha$ for some $j=1, \ldots, m$ and $\alpha \in H_2^{\mathrm{eff}}(X)$ with $c_1(\alpha) = 0$.
\end{lem}
\begin{proof}
For a toric manifold $X$, the classification result of Cho-Oh \cite{cho06} says that a smooth non-constant holomorphic disc bounded by a Lagrangian torus fiber has Maslov index at least 2, and one with Maslov index equal to 2 must represent a basic disc class $\beta_j$ for some $j = 1, \ldots, m$.  If $X$ is semi-Fano, every holomorphic curve has non-negative Chern number.  Now a non-constant stable disc bounded by a Lagrangian torus fiber consists of at least one holomorphic disc component and possibly several sphere components.  Thus it has Maslov index at least 2, and if it is of Maslov index 2, it must consist of only one disc component which represents a basic disc class $\beta_j$.  Moreover, the sphere components all have Chern number zero, and thus they are contained in the toric divisors (otherwise they are constant and cannot be stable since there is only one interior marked point).  But a holomorphic disc in class $\beta_j$ intersect with $\cup_k D_k$ at only one point.  Thus it is only attached with one of the sphere components, and by connectedness the sphere components form a rational curve whose class is denoted as $\alpha$ which has $c_1 = 0$.
\end{proof}

By Cho-Oh \cite{cho06}, $n_1(\beta_j) = 1$ for $j=1,\ldots,m$. Hence the disc potential is of the form:
\begin{defn}[Disc potential \cite{FOOO1}] \label{defn disc potential}
For a semi-Fano toric manifold $X$, the {\em disc potential} of $X$ is defined by
$$ W^{\mathrm{LF}} := \sum_{l=1}^{m} (1+\delta_l) Z_l, $$
where
\begin{equation} \label{Eqn Z}
Z_l = \left\{ \begin{array}{ll}
z_l & \text{ when } l=1,\ldots,n;\\
q_{l-n} z^{v_l}:=q_{l-n}\prod_{i=1}^n z_i^{(\nu_i, v_l)} & \text{ when } l=n+1,\ldots,m,
\end{array}
\right.
\end{equation}
$$ \delta_l := \sum_{\alpha \in H_2^{c_1=0}(X)\setminus\{0\}} n_1(\beta_l+\alpha) q^{\alpha}, \quad \text{ for all } l=1,\ldots,m,$$ $q^{\alpha} = \prod_{k=1}^{m-n} q_k^{D_{n+k} \cdot \alpha}$, and $H_2^{c_1=0}(X)\subset H_2^{\eff}(X)$ denotes the semi-group of all effective curve classes $\alpha$ with $c_1(\alpha) = 0$. We also call $W^{\mathrm{LF}}$ the {\em Lagrangian Floer superpotential} of $X$.
\end{defn}

$\delta_l$ can also be expressed in terms of the flat coordinates $q^\nef$ defined using a nef basis of $H^2(X)$.  A priori each $\delta_l$ is only a formal power series in the formal Novikov variables $q^\nef_1, \ldots, q^\nef_k$. In this paper we will show that $W^{\mathrm{LF}}$ is equal to the {\em Hori-Vafa superpotential} via the inverse mirror map and it will follow that each $\delta_l$ is in fact a convergent power series.

In Floer-theoretic terms, $W^{\mathrm{LF}}$ is exactly the $m_0$-term, which, for toric manifolds, governs the whole Lagrangian Floer theory. All the higher $A_\infty$-products $m_k$, $k\geq 1$ can be recovered by taking the derivatives of $m_0$, and it can be used to detect the non-displaceable Lagrangian torus fibers.  See \cite{FOOO1,FOOO2,FOOO10b} for detailed discussions.


\section{Hori-Vafa superpotential and the toric mirror theorem} \label{sect tor mir}
We now come to the complex geometry (B-model) of mirrors of toric manifolds.  The mirror of a toric variety $X$ is given by a Laurent polynomial $W^{\mathrm{HV}}$ which is explicitly determined by the fan of $X$ \cite{givental98, hori00}.  It defines a singularity theory whose moduli has flat coordinates given by the oscillatory integrals.  These have explicit formulas and will be reviewed in this section.  We will then recall the celebrated mirror theorem for toric varieties \cite{givental98, LLY3}.

\subsection{Mirror theorems}
The mirror complex moduli is defined as a certain neighborhood of $0$ of $(\cpx^*)^{m-n}$ (see Definition \ref{defn mir mod}), whose coordinates are denoted as $\check{q} = (\check{q}_1, \ldots, \check{q}_{m-n})$; $\check{q}_k$ is also denoted as $\check{q}^{\Psi_k}$ for $k=1,\ldots,m-n$.  We may also use a nef basis of $H^2(X)$ instead, and the corresponding complex coordinates are denoted as $\check{q}_k^\nef$.

\begin{defn}[$I$-function] \label{defn I}
The {\em $I$-function} of a toric manifold $X$ is defined as
$$ I^X(\check{q},z) := \exp\left(\frac{1}{z} \sum_{k=1}^{m-n} (\log \check{q}_k)[D_{n+k}] \right) \sum_{d \in H_2^{\eff}(X)} \check{q}^d \, I_d $$
where
$$ I_d := \prod_{l=1}^m\frac{\prod_{s=-\infty}^0(D_l+sz)}{\prod_{s=-\infty}^{D_i\cdot d}(D_l+sz)}$$
and $\check{q}^d := \prod_{k=1}^{m-n} \check{q}_k^{D_{n+k} \cdot d}$.
\end{defn}

Notice that in the above expression of $I$, $z$ is the $\mathbb{S}^1$-equivariant parameter.  By doing a Laurent expansion around $z=\infty$, we see that $I_d$ can be regarded as a $\Sym^* (H_\cpx^2(X,\tor))$-valued function (or as an element of $\Sym^* (H^2(X,T)) ((z^{-1})) $), where $H_\cpx^2(X,\tor) := H^2(X,\tor) \otimes \cpx$.  A basis of $H^2(X,\tor)$ is given by $\{D_l\}_{l=1}^m$, which is dual to the basis  $\{\beta_l\}_{l=1}^m \subset H_2(X,\tor)$.  The canonical projection $H^2(X,\tor) \to H^2(X)$ sends $D_l$ to its class $[D_l]$ for $l=1,\ldots,m$.  Since $\{[D_l]\}_{l=n+1}^m$ forms a basis of $H^2(X,\cpx)$, we may choose a splitting $H^2(X,\cpx) \hookrightarrow H_\cpx^2(X,\tor)$ by taking the basic vector $[D_{n+k}]$ to $D_{n+k}$ for $k=1, \ldots, m-n$.  In this way we can regard $H^2(X,\cpx)$ as a subspace of $H_\cpx^2(X,\tor)$.

\begin{defn}[The mirror map] \label{Defn mir map}
Let $X$ be a semi-Fano toric manifold.  The {\em mirror map} is defined as the $1/z$-coefficient of the $I$-function of $X$, which is an $H^2(X,\cpx)$-valued function in $\check{q} \in (\cpx^*)^{m-n}$. More precisely, the $1/z$-coefficient of $\sum_{d \in H_2^{\eff}(X)} \check{q}^d I_d$ is of the form
$$ -\sum_{l=1}^m g_l(\check{q}) D_l \in H_\cpx^2(X,\tor)$$
where the functions $g_l$ in $\check{q} \in (\cpx^*)^{m-n}$, $l=1,\ldots,m$ are given by (\ref{eqn:funcn_g}).
Then we can write
$$ \sum_{k=1}^{m-n} (\log \check{q}_k) [D_{n+k}] - \sum_{l=1}^m g_l(\check{q}) [D_l] = \sum_{k=1}^{m-n} (\log \check{q}_k - g^{\Psi_k}(\check{q})) [D_{n+k}],$$
where
\begin{equation*}
g^{\Psi_k} := \sum_{l=1}^m \left( D_l \cdot \Psi_k \right) g_l.
\end{equation*}
Thus in terms of the coordinates $(q_k)_{k=1}^{m-n}$ of $H^2(X,\cpx)/2\pi\consti H^2(X,\integer)$, the mirror map $q(\check{q})$ is
\begin{equation}\label{mirror_map_eqn}
q_k(\check{q}) = \check{q}_k \exp (- g^{\Psi_k}(\check{q}))
\end{equation}
for $k=1,\ldots,m-n$.
\end{defn}

Each $g_l$ can also be expressed in terms of the flat coordinates $\check{q}^\nef$ defined by the dual of a nef basis of $H^2(X)$.  While a priori $g_l(\check{q}^\nef)$ is a formal power series in $\check{q}^\nef$ (or an element in the Novikov ring), by the theory of hypergeometric series it is known that $g_l(\check{q}^\nef)$ is indeed convergent around $\check{q}^\nef = 0$.  Moreover, the mirror map $q^\nef(\check{q}^\nef)$ is a local diffeomorphism, and its inverse is denoted as $\check{q}^\nef(q^\nef)$.

\begin{defn} \label{defn mir mod}
The {\em mirror complex moduli} $\moduli^{\mathrm{mir}}$ is defined as the domain of convergence of $(g_l(\check{q}^\nef))_{l=1}^m$ around $\check{q}^\nef = 0$ in $(\cpx^*)^{m-n}$.

The {\em K\"ahler moduli} is defined as the intersection of the complexified K\"ahler cone $\Kcone^{\cpx}_X$ with the domain of convergence of the inverse mirror map $\check{q}^\nef(q^\nef)$.  By abuse of notation we will still denote the K\"ahler moduli as $\Kcone^{\cpx}_X$.
\end{defn}

The most important result in closed-string mirror symmetry for toric manifolds is:
\begin{thm}[Toric mirror theorem \cite{givental98, LLY3}] \label{thm toric mir thm}
Let $X$ be a compact semi-Fano toric manifold. Consider the {\em $J$-function} of $X$:
\begin{equation} \label{J-function}
J(q,z)= \exp\left(\frac{1}{z} \sum_{k=1}^{m-n} (\log q_k)[D_{n+k}] \right) \left(1+\sum_a\sum_{d \in H_2^\mathrm{eff}(X)\setminus\{0\}} q^d \Big\langle1,\frac{\Phi_a}{z-\psi}\Big\rangle_{0,2,d}\Phi^a\right)
\end{equation}
where $\{\Phi_a\}$ is a homogeneous additive basis of $H^*(X)$ and $\{\Phi^a\} \subset H^*(X)$ is its dual basis with respect to the Poincar\'e pairing.  We always use $\langle \cdots \rangle_{g,k,d}$ to denote the genus $g$, degree $d$ descendent GW invariant of $X$ with $k$ insertions. Then
$$ I(\check{q},z) = J(q(\check{q}),z) $$
where $q(\check{q})$ is the mirror map given in Definition \ref{Defn mir map}.
\end{thm}

In this paper, we are interested in open-string mirror symmetry. The Hori-Vafa superpotential (which plays a role analogous to that of the $I$-function in closed-string mirror symmetry) is the central object for this purpose.

\begin{defn} \label{defn HV}
The {\em Hori-Vafa superpotential} of the toric manifold $X$ is a holomorphic function $W^{\mathrm{HV}}: \moduli^{\mathrm{mir}} \times (\cpx^*)^n \to \cpx^*$ defined by
$$ W^{\mathrm{HV}}_{\check{q}}(z_1, \ldots, z_n) = z_1 + \cdots + z_n + \sum_{k=1}^{m-n} \check{q}_{k}z^{v_{n+k}} $$
where $z^{v_{n+k}}$ denotes the monomial $\prod_{l=1}^n z_l^{\pairing{v_{n+k}}{\nu_l}}$.  It is pulled back to the K\"ahler moduli $\Kcone^\cpx_X$ by substituting the inverse mirror map $\check{q}_{k} = \check{q}_{k}(q)$ ($k=1,\ldots,m-n$) in the above expression.
\end{defn}

The Hori-Vafa potential may also be written as
\begin{equation} \label{eqn W tilde}
\tilde{W}^{\mathrm{HV}}_{\check{q}}(z_1, \ldots, z_n) = \sum_{p=1}^n (\exp g_p(\check{q})) z_p +
\sum_{k=1}^{m-n} \check{q}_{k} z^{v_{n+k}} \prod_{p=1}^n \exp \left(\pairing{v_{n+k}}{\nu_p}g_p(\check{q})\right).
\end{equation}
via the coordinate change $z_p \mapsto (\exp g_p(\check{q})) z_p$, $p=1,\ldots,n$.  Such a coordinate change will be necessary for the comparison with the disc potential.

The Hori-Vafa superpotential contains closed-string enumerative information of $X$:
\begin{thm}[Second form of the toric mirror theorem \cite{givental98, LLY3}] \label{Giv_thm2}
Let $X$ be a semi-Fano toric manifold and $\omega \in \Kcone^\cpx_X$ with coordinate $q(-\omega) = q$. Let $W^{\mathrm{HV}}$ be its Hori-Vafa superpotential. Then
$$\QH^* (X,\omega) \cong \Jac (W^{\mathrm{HV}}_{\check{q}(q)})$$
where $\check{q}(q)$ is the inverse mirror map.
Moreover, the isomorphism is given by sending the generators $[D_{n+k}] \in \QH^* (X,\omega)$ to $\left[\dd{\log q_{k}} W^{\mathrm{HV}}_{\check{q}(q)}\right] \in \Jac (W^{\mathrm{HV}}_{\check{q}(q)})$.
\end{thm}

\subsection{Extended moduli} \label{sec:extmod}
We have seen that $I_d$ is indeed $\Sym^*(H_\cpx^2(X,\tor))$-valued.  Thus it is natural to extend the mirror map and the Hori-Vafa superpotential $W^{\mathrm{HV}}_{\check{q}}$ from $H_\cpx^2(X)$ to $H_\cpx^2(X,\tor)$.  Extended moduli was introduced by Givental \cite{givental98} (see also Iritani \cite{iritani09}).

Let
$$Q_l = \exp\left( - \pairing{\cdot}{\beta_l} \right), \quad\text{for }l=1,\ldots, m,$$
be the flat coordinates on $H_\cpx^2(X,\tor)/2\pi\consti H^2(X,\tor)$.  In terms of these coordinates, the canonical projection
$$H_\cpx^2(X,\tor)/2\pi\consti H^2(X,\tor) \to H^2(X,\cpx)/2\pi\consti H^2(X,\integer)$$
is given by $q_k = Q^{-v_{n+k}} Q_{n+k}$ for $k=1,\ldots,m-n$.  Here $Q^{-v_{n+k}}$ denotes the monomial $\prod_{l=1}^n Q_l^{\pairing{-v_{n+k}}{\nu_l}}$.
The splitting $$H^2(X,\cpx)/2\pi\consti H^2(X,\integer) \hookrightarrow H_\cpx^2(X,\tor)/2\pi\consti H^2(X,\tor)$$ is given by $Q_l = 1$  for $l = 1, \ldots, n$ and $Q_l = q_l$ for $l=n+1,\ldots,m$.

\begin{defn}[Extended moduli] \label{Defn ext_mod}
The {\em extended K\"ahler moduli}
$\tilde{\Kcone}^{\cpx}_X\subset H_\cpx^2(X,\tor)/2\pi\consti H^2(X,\tor)$
is defined as the inverse image of the K\"ahler moduli $\Kcone^{\cpx}_X \subset H^2(X,\cpx)/2\pi\consti H^2(X,\integer)$ under the canonical projection $$H_\cpx^2(X,\tor)/2\pi\consti H^2(X,\tor) \to H^2(X,\cpx)/2\pi\consti H^2(X,\integer).$$

The {\em extended mirror complex moduli} $\tilde{\moduli}^{\mathrm{mir}}$ is the inverse image of $\moduli^{\mathrm{mir}} \subset (\cpx^*)^{m-n}$ under the projection $(\cpx^*)^{m} \to (\cpx^*)^{m-n}$ defined by sending $(\check{Q}_1, \ldots, \check{Q}_m) \in (\cpx^*)^{m}$ to $(\check{q}_1, \ldots, \check{q}_{m-n})\in(\cpx^*)^{m-n}$, $\check{q}_k = \check{Q}^{-v_{n+k}} \check{Q}_{n+k}$ for $k=1,\ldots,m-n$.  $\moduli^{\mathrm{mir}}$ can be regarded as a submanifold of $\tilde{\moduli}^{\mathrm{mir}}$ by the splitting
$$\moduli^{\mathrm{mir}} \hookrightarrow \tilde{\moduli}^{\mathrm{mir}} $$
given by $\check{Q}_l = 1$ for $l=1,\ldots,n$ and $\check{Q}_{n+k} = \check{q}_{k}$ for $k=1,\ldots,m-n$.
\end{defn}

\begin{defn}[Extended superpotential and mirror map] \label{Def ext}
The {\em extension of Hori-Vafa potential} $W^{\mathrm{HV}}_{\check{q}}$ from $\check{q} \in \moduli^{\mathrm{mir}}$ to $\check{Q} \in \tilde{\moduli}^{\mathrm{mir}}$ is defined as
$$\mathscr{W}^{\mathrm{HV}}_{\check{Q}}(z_1,\ldots,z_n) := \sum_{l=1}^m \check{Q}_l z^{v_l}$$
where $z^{v_l}$ denotes the monomial $\prod_{p=1}^n z_p^{\pairing{v_l}{\nu_p}}$ (and so $z^{v_l} = z_l$ for $l = 1, \ldots, n$).

The {\em extended mirror map} from $\tilde{\moduli}^{\mathrm{mir}}$ to $\tilde{\Kcone}^{\cpx}_X$ is defined to be
$$ Q_l(\check{Q}) = \check{Q}_l \exp (-g_l(\check{q}(\check{Q})))$$
where $\check{q}(\check{Q})$ is the canonical projection $\check{q}_k = \check{Q}^{-v_{n+k}} \check{Q}_{n+k}$ for $k=1,\ldots,m-n$.

The {\em extended inverse mirror map} from $\tilde{\Kcone}^{\cpx}_X$ to $\tilde{\moduli}^{\mathrm{mir}}$ is defined to be
$$ \check{Q}_l(Q) = Q_l \exp (g_l(\check{q}(q(Q))))$$
where $q(Q)$ is the canonical projection $q_k = Q^{-v_{n+k}} Q_{n+k}$ for $k=1,\ldots,m-n$, and $\check{q}(q)$ is the inverse mirror map.
\end{defn}


The following proposition follows immediately from the above definitions:
\begin{prop} \label{Prop compatible}
We have the following commutative diagram
$$
\begin{diagram}
	\node{\tilde{\moduli}^{\mathrm{mir}}} \arrow{s} \arrow{e} \node{\tilde{\Kcone}_X^\C} \arrow{s} \\
	\node{\moduli^{\mathrm{mir}}} \arrow{e} \node{\Kcone^\C_X}
\end{diagram}
$$
and a similar one for the inverse mirror map and its extended version.
As a consequence,
$$ W^{\mathrm{HV}}_{\check{q}(q)} = W^{\mathrm{HV}}_{\check{q}(\check{Q}(q))} =\mathscr{W}^{\mathrm{HV}}_{\check{Q}(q)}$$
where $\check{q}(q)$ is the inverse mirror map, $\check{Q}(q)$ is the restriction of the extended inverse mirror map $\check{Q}(Q)$ to the K\"ahler moduli $\Kcone^\cpx \ni q$, and $\check{q}(\check{Q})$ is the canonical projection.
\end{prop}

\begin{defn}\label{defn_diff_operator}
An element $A = \sum_{l=1}^m \tilde{a}_l D_l \in H_\cpx^2(X,\tor)$ induces a differential operator
$$\hat{A}_{\check{Q}} := \sum_{l=1}^m \tilde{a}_l \partial_{\log\check{Q}_l}$$
which operates on functions on $\tilde{\moduli}^{\mathrm{mir}}$; such an association is linear, i.e. $\widehat{(c A + B)}_{\check{Q}} = c \hat{A}_{\check{Q}} + \hat{B}_{\check{Q}}$ for all $c \in \cpx$.  The element $A$ projects to $[A] \in H^2(X,\cpx)$ which can be written as $\sum_{k=1}^{m-n} a_k [D_{n+k}]$ in terms of the basis $\{[D_{n+k}]\}_{k=1}^{m-n}$.  It induces the differential operator
$$\hat{A}_{\check{q}} := \sum_{k=1}^{m-n} a_k \partial_{\log\check{q}_k}$$
which operates on functions on $\moduli^{\mathrm{mir}}$;
similarly this association is linear.

Replacing $\check{Q}$ by $Q$ and $\check{q}$ by $q$, $A$ induces the differential operator on the extended K\"ahler moduli $\tilde{\Kcone}^{\cpx}$:
$$\hat{A}_{Q} := \sum_{l=1}^m \tilde{a}_l \partial_{\log Q_l}$$
and the differential operator on the K\"ahler moduli $\Kcone^{\cpx}$:
$$\hat{A}_{q} := \sum_{k=1}^{m-n} a_k \partial_{\log q_k}.$$
\end{defn}

A good thing about the extension $\mathscr{W}^{\mathrm{HV}}$ is the following observation:

\begin{prop} \label{Prop hat}
If $A,B \in H_\cpx^2(X,\tor)$ project to the same element in $H^2(X,\cpx)$ (meaning that $A$ and $B$ are linearly equivalent), then
$$ [\hat{A}_{\check{Q}} \mathscr{W}^{\mathrm{HV}}] = [\hat{B}_{\check{Q}} \mathscr{W}^{\mathrm{HV}}] \in \Jac(\mathscr{W}^{\mathrm{HV}}_{\check{Q}}). $$
In particular, restricting to the mirror moduli $\moduli^{\mathrm{mir}}$, one has
$$ [(\hat{A}_{\check{Q}} \mathscr{W}^{\mathrm{HV}})(\check{q}))] = [\hat{A}_{\check{q}} W^{\mathrm{HV}}] \in \Jac(W^{\mathrm{HV}}_{\check{q}}). $$
\end{prop}
\begin{proof}
For the first statement, it suffices to prove that if $\sum_{l=1}^m \tilde{a}_l D_l$ is linearly equivalent to zero, then
$ \sum_{l=1}^m \tilde{a}_l \partial_{\log\check{Q}_l} \mathscr{W}^{\mathrm{HV}}_{\check{Q}} $
is in the Jacobian ideal of $\mathscr{W}^{\mathrm{HV}}_{\check{Q}}$. Now
$$ \sum_{l=1}^m \tilde{a}_l \partial_{\log\check{Q}_l} \mathscr{W}^{\mathrm{HV}}_{\check{Q}} = \sum_{l=1}^m \tilde{a}_l \check{Q}_l z^{v_l}. $$
Since $\sum_{l=1}^m \tilde{a}_l D_l$ is linearly equivalent to zero, there exists $\nu \in M$ such that $\pairing{\nu}{v_l} = \tilde{a}_l$ for all $l = 1, \ldots, m$. So the above expression is equal to
$$\sum_{l=1}^m \pairing{\nu}{v_l} \check{Q}_l z^{v_l} = \sum_{k=1}^n (\nu, v_k) \frac{\partial}{\partial \log z_k} \mathscr{W}^{\mathrm{HV}}_{\check{Q}}.$$

For the second statement, write $A = \sum_{l=1}^m \tilde{a}_l D_l$ and its projection $[A] = \sum_{k=1}^{m-n} a_k [D_{n+k}]$.  Then $A$ and $\sum_{k=1}^{m-n} a_k D_{n+k} \in H_\cpx^2(X,\tor)$ projects to the same element in $H^2(X,\cpx)$.  Thus
$$ [\hat{A}_{\check{Q}} \mathscr{W}^{\mathrm{HV}}] = \left[\sum_{k=1}^{m-n} a_k \partial_{\log\check{Q}_{n+k}} \mathscr{W}^{\mathrm{HV}}\right] \in \Jac(\mathscr{W}^{\mathrm{HV}}_{\check{Q}}). $$
Take  $\check{Q}_l = 1$ for $l=1,\ldots,n$ and $\check{Q}_{n+k} = \check{q}_k$ for $k=1,\ldots,m-n$, since $\mathscr{W}^{\mathrm{HV}}_{\check{q}} = W^{\mathrm{HV}}_{\check{q}}$, we get
$$ [(\hat{A}_{\check{Q}} \mathscr{W}^{\mathrm{HV}})(\check{q}))] = [\hat{A}_{\check{q}} W^{\mathrm{HV}}] \in \Jac(W^{\mathrm{HV}}_{\check{q}}). $$
\end{proof}

\subsection{Batyrev elements}
Theorem \ref{Giv_thm2} gives a presentation of the quantum cohomology ring, where the generators are given by the {\em Batyrev elements} introduced by Gonz\'alez-Iritani \cite{G-I11}.  
\begin{defn}[Batyrev elements \cite{G-I11}]
The {\em Batyrev elements}, which are $H^2(X,\cpx)$-valued functions on $\moduli^{\mathrm{mir}}$, are defined as follows.  For $k=1,\ldots,m-n$,
$$ B_{n+k} := \sum_{r=1}^{m-n} \frac{\partial \log q_{r}(\check{q})}{\partial \log \check{q}_{k}} [D_{n+r}] $$
where $q(\check{q})$ is the mirror map.
For $l=1,\ldots,n$,
$$ B_l := \sum_{k=1}^{m-n} \left(D_l \cdot \Psi_k \right) B_{n+k}.$$
\end{defn}

The Batyrev elements satisfy two sets of explicit relations \cite{G-I11}:
\begin{enumerate}
\item
Linear relations.  It follows from the definition that $\{B_l\}_{l=1}^m$ satisfies the same linear relations as that satisfied by $\{D_l\}_{l=1}^m$, namely, for every $\nu \in M$,
$$ \sum_{l=1}^m \pairing{\nu}{v_l} B_l = 0. $$

\item
Multiplicative relations.  For every $k=1,\ldots,m-n$,
\begin{equation}\label{eqn Batyrev multi}
B_1^{D_1 \cdot \Psi_k} * \cdots * B_m^{D_m \cdot \Psi_k} = q_k,
\end{equation}
where $B_l^{D_l \cdot \Psi_k}$ means $B_l$ quantum-multiplies itself for $D_l \cdot \Psi_k$ times.  This relation is a consequence of the toric mirror theorem (second form, see Theorem \ref{Giv_thm2}).
\end{enumerate}

Batyrev elements can also be lifted to $\tilde{\moduli}^{\mathrm{mir}}$:
\begin{defn}[Extended Batyrev elements] \label{def:extB}
Define the following $H^2_\cpx (X,\tor)$-valued functions on $\tilde{\moduli}^{\mathrm{mir}}$:
$$ \tilde{B}_l := \sum_{p=1}^m \frac{\partial \log Q_p(\check{Q})}{\partial \log \check{Q}_l} D_p, \quad l=1,\ldots,m$$
Here $Q(\check{Q})$ is the extended mirror map given in Definition \ref{Def ext}.
\end{defn}

More conceptually, the extended Batyrev elements are push-forward of the vector fields $\dd{\log \check{Q}_l}$ for $l=1,\ldots,m$ via the extended mirror map $\tilde{\moduli}^{\mathrm{mir}} \to \tilde{\Kcone}_X^\C$, and Batyrev elements are push-forward of the vector fields $\dd{\log \check{q}_l}$ for $l=1,\ldots,m-n$ via the mirror map $\moduli^{\mathrm{mir}} \to \Kcone_X^\C$.  So by the commutative diagram of Proposition \ref{Prop compatible}, we have
\begin{prop} \label{Prop [B]}
$[\tilde{B}_l|_{\moduli^{\mathrm{mir}}}] = B_l \text{ for } l=1,\ldots,m.$
\end{prop}

It follows from the above discussions that Batyrev elements have a very simple form under the isomorphism $\QH^* (X,\omega_q) \cong \Jac(W^{\mathrm{HV}}_{\check{q}(q)})$:
\begin{prop} \label{prop Bat to}
Under the isomorphism $\QH^* (X,\omega_q) \cong \Jac(W^{\mathrm{HV}}_{\check{q}(q)})$ of Theorem \ref{Giv_thm2}, the Batyrev elements $B_l$ are mapped to $z^{v_l}$ for $l=1,\ldots,n$, and $\check{q}_{l-n}(q) z^{v_l}$ for $l=n+1,\ldots,m$. Equivalently, each $B_l$ is mapped to $(\exp g_l(\check{q}(q))) Z_l$ for $l=1,\ldots,m$ under $\QH^* (X,\omega_q) \cong \Jac(\tilde{W}^{\mathrm{HV}}_{\check{q}(q)})$, where $Z_l$ is defined by Equation \eqref{Eqn Z}.
\end{prop}
\begin{proof}
Associate $D_p \in H^2_\cpx(X,\tor)$ to the differential operator $\dd{\log Q_p}$ for $p=1,\ldots,m$.  Then $\tilde{B}_l$ is associated with
$\widehat{\tilde{B}_l} = \sum_{p=1}^m \frac{\partial \log Q_p}{\partial \log \check{Q}_l} \dd{\log Q_p} = \dd{\log \check{Q}_l}$.
Restricting to $\moduli^{\mathrm{mir}}$, by Proposition \ref{Prop [B]}, $\tilde{B}_l$ projects to $B_l \in H^2(X,\cpx)$.  By Proposition \ref{Prop hat},
$[\widehat{\tilde{B}_l} \mathscr{W}^{\mathrm{HV}}(\check{q}(q))] = [\widehat{B_l} W^{\mathrm{HV}}_{\check{q}(q)}] \in \Jac(W^{\mathrm{HV}}_{\check{q}(q)})$.
On the other hand,
$ \widehat{\tilde{B}_l} \mathscr{W}^{\mathrm{HV}}(\check{q}(q)) = \dd{\log \check{Q}_l} \sum_{p=1}^m \check{Q}_p z^{v_p} = \check{Q}_l z^{v_l}$.
Thus
$[\widehat{B_l} W^{\mathrm{HV}}_{\check{q}(q)}] = z^{v_l}$
for $l=1,\ldots,n$ and
$[\widehat{B_{n+k}} W^{\mathrm{HV}}_{\check{q}(q)}] = \check{q_k}(q) z^{v_{n+k}}$
for $k=1,\ldots,m-n$.
Under the coordinate change \eqref{eqn W tilde}, $z_l$ changes to $(\exp g_l(\check{q}(q))) z_l$ for $l=1,\ldots,n$, and for $l=n+1,\ldots,m$, $\check{q}_{l-n}(q) z^{v_l}$ changes to
\begin{align*}
\check{q}_{l-n}(q) \prod_{p=1}^n \big((\exp g_p(\check{q}(q))) z_p\big)^{\pairing{v_l}{\nu_p}} &= z^{v_l} q_{l-n} (\exp g^{\Psi_{l-n}}(\check{q}(q)))  \prod_{p=1}^n (\exp g_p(\check{q}(q))) \\
&= q_{l-n} (\exp g_l(\check{q}(q))) z^{v_l}.
\end{align*}
\end{proof}

\section{Seidel representations for toric manifolds} \label{Sect Seidel_rep}
In this section we review the construction and properties of the Seidel representation \cite{seidel97, McDuff_seidel}, which is an action\footnote{Here $\mathrm{Ham}(X,\omega)$ denotes the group of Hamiltonian diffeomorphisms of $(X,\omega)$.} of $\pi_1(\mathrm{Ham}(X,\omega))$ on $\QH^*(X,\omega)$, in the toric case. A key insight of this paper is that open GW invariants of a semi-Fano toric manifold $X$ are equal to some closed GW invariants of certain manifolds related to $X$ used to construct these representations, and so we call them the Seidel spaces:
\begin{defn}
Let $X$ be a manifold. Suppose that we have an action $\rho:\cpx^* \times X \to X$ of $\cpx^*$ on $X$.  The manifold
$$E = E_\rho := (X \times (\cpx^2 \setminus \{0\})) / \cpx^*$$
is called the {\em Seidel space} associated to the action $\rho$, where $z \in \cpx^*$ acts on the second factor $\cpx^2 \setminus \{0\} \ni (u,v)$ by $z \cdot (u,v) = (zu,zv)$. The Seidel space $E$ is an $X$-bundle over $\proj^1$ where the bundle map $(X \times (\cpx^2 \setminus \{0\})) / \cpx^* \to \proj^1$ is given by the projection to the second factor.
\end{defn}

Let  $X$ be a toric $n$-fold defined by a fan $\Sigma^X$ supported on the vector space $N\otimes_\mathbb{Z} \mathbb{R}$ where $N$ is a lattice. Each lattice point $v \in N$ produces a $\cpx^*$-action on $X$, which can be written as $t \cdot [a+\consti b] = [a + \consti b + v \log t]$ for $[a+\consti b] \in N_\C / (2\pi\consti N) \subset X$.

In particular, the minimal generator $v_j\in N$ of a ray of $\Sigma^X$ gives a $\cpx^*$-action and thus defines a corresponding Seidel space $E=E_j$.  It is a toric manifold of dimension $n+1$ whose fan $\Sigma^{E}$ has rays generated by $v_l^E = (0,v_l)$ for $l = 1, \ldots, m$, $v_0^E = (1,0)$ and $v_\infty^E = (-1,v_j)$.

On the other hand, we may use the opposite direction $-v_j \in N$ to generate a $\cpx^*$-action. The corresponding Seidel space will be denoted by $E^- = E_j^-$.  It is also a toric manifold of dimension $n+1$ whose fan $\Sigma^{E^-}$ has rays generated by $v_l^{E^-} = (0,v_l)$ for $l = 1, \ldots, m$, $v_0^{E^-} = (1,0)$ and $v_\infty^{E^-} = (-1,-v_j)$.

Since $E^-$ is a toric manifold,  $\pi_2(E^-,\tor^{E^-})$ is generated by the basic disc classes which are denoted as $b_l$, $l = 0, 1, \ldots, m, \infty$ (while recall that basic disc classes in $X$ are denoted as $\beta_l$ for $l = 1,\ldots,m$).  Moreover, the toric prime divisors of $E^-$ are denoted as $\mathscr{D}_0, \mathscr{D}_1, \ldots, \mathscr{D}_m, \mathscr{D}_\infty$ (while recall that toric prime divisors of $X$ are denoted as $D_l$ for $l = 1,\ldots,m$).
Viewing Seidel spaces as $X$-bundles over $\mathbb{P}^1$, one has the following specific sections of the Seidel spaces:

\begin{defn} \label{defn section}
Let $X$ be a toric manifold, and let $v_j\in N$ be the minimal generator of a ray in the fan of $X$. Let $E=E_j$ and $E^-=E_j^-$ be the Seidel spaces associated to $v_j$ and $-v_j$ respectively.  Under the $\cpx^*$-action generated by either $v_j$ or $-v_j$, there are finitely many fixed loci in $X$.  One of them is $D_j \subset X$, whose normal bundle is of rank one with weight $-1$ with respect to $v_j$ (or weight $1$ with respect to $-v_j$).  Each point $p \in D_j$ gives a section $\sigma=\sigma_j: \proj^1 \to E$ (resp. $\sigma^-=\sigma_j^-: \proj^1 \to E^-$) whose value is constantly $p \in D_j \subset X$.  It is called the {\em zero section} of $E$ (resp. $E^-$).

There is another unique fixed locus $S$ in $X$ whose normal bundle has all weights positive with respect to the $\cpx^*$-action of $v_j$ (or all negative with respect to $-v_j$).  Similarly each point $p \in S$ gives a section $\sigma_\infty: \proj^1 \to E$ (resp. $\sigma^-_\infty: \proj^1 \to E^-$) whose value is constantly $p \in S \subset X$, and it is called an {\em infinity section} of $E$ (resp. $E^-$).
\end{defn}

The various sections in the above definition are illustrated by Figure \ref{fig Seidel_sp} below, which depicts the Seidel spaces of $\proj^1$.
\begin{figure}[htp]
\begin{center}
\includegraphics[scale=0.8]{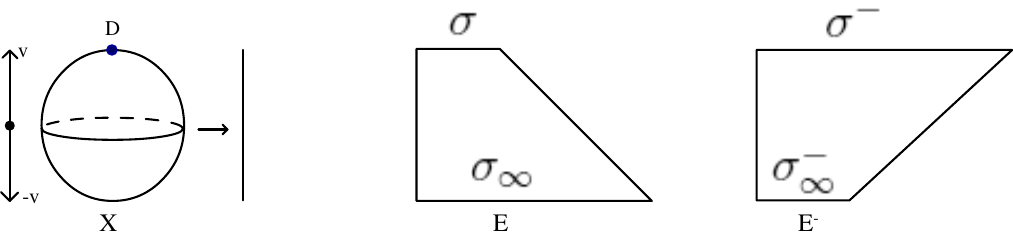}
\end{center}
\caption{This figure shows the Seidel spaces $E$ and $E^-$ associated to the two different torus actions on $\proj^1$ and their sections corresponding to the fixed points in $\proj^1$ under the action.  In this simple case both $E$ and $E^-$ are the Hirzebruch surface $\mathbb{F}_1$, while in general they are different manifolds.}
\label{fig Seidel_sp}
\end{figure}
By abuse of notation their classes in $H_2(E)$ and $H_2(E^-)$ are also denoted as $\sigma$, $\sigma_\infty$ and $\sigma^-$, $\sigma_\infty^-$ respectively.  Notice that under the $\cpx^*$-action generated by $v_j$, $D_j$ has negative weight (namely, $-1$); while under the $\cpx^*$-action generated by $-v_j$, the other fixed locus $S$ has negative weight.  Then using Lemma 2.2 of Gonz\'alez-Iritani \cite{G-I11}, all the curve classes of $E$ ($E^-$ resp.) are generated by $\sigma$ ($\sigma_\infty^-$ resp.) and curve classes of $X$:

\begin{prop}[Lemma 2.2 of \cite{G-I11}] \label{Prop sigma gen} We have
$$ H_2^{\eff}(E) = \integer_{\geq 0}[\sigma] + H_2^{\eff}(X), \quad H_2^{\eff}(E^-) = \integer_{\geq 0}[\sigma_\infty^-] + H_2^{\eff}(X),$$
where $H_2^{\eff}(E)$,  $H_2^{\eff}(E^-)$, and $H_2^{\eff}(X)$ denote the Mori cones of effective curve classes of $E$, $E^-$, and $X$ respectively. 
\end{prop}

The section class $\sigma^- \in H_2^\eff(E^-)$ can also be written as $\sigma^- = b_0 + b_\infty + b_j$, where we recall that $b_l$ for $l = 0, 1, \ldots, m, \infty$ are the basic disc classes of $E^-$.  It is the most important curve class to us as we shall see in the next section.  By Proposition \ref{Prop sigma gen}, it can also be written as
$$ \sigma^- = \sigma^-_{\infty} + \uf $$
for some curve class $\uf$ in $X$.

\begin{defn}[Seidel element]\label{defn:seidel_elt}
Given a $\cpx^*$-action on $X$, let $E$ be the corresponding Seidel space.  The {\em Seidel element} is
$$ S := \sum_{a} \sum_{d\in H_2^{ \mathrm{sec}}(E)} q^{d} \langle \Phi_a^E \rangle_{0,1,d} \, \Phi^a = q^{\sigma} \sum_{a} \sum_{\alpha \in H_2^{ \mathrm{eff}}(X)} q^{\alpha} \langle \Phi_a^E \rangle_{0,1,\sigma + \alpha} \, \Phi^a, $$
where $\{\Phi_a\}$ is a basis of $H^*(X)$ and $\{\Phi^a\}$ is the dual basis with respect to Poincar\'e pairing;  $\Phi_a^E \in H^*(E)$ denotes the push-forward of $\Phi_a$ under the inclusion of $X$ into $E$ as a fiber;  and
$ H_2^{\mathrm{sec}}(E) := \{\sigma + \alpha: \alpha \in H_2^{\eff}(X)\} \subset H_2^{\eff}(E)$,
where $\sigma$ is the section class of $E$ corresponding to the fixed locus in $X$ with all weights to be negative.
The {\em normalized Seidel element} is
$$ S^\circ := \sum_{a} \sum_{\alpha \in H_2^{ \mathrm{eff}}(X)} q^{\alpha} \langle \Phi_a^E \rangle_{0,1,\sigma + \alpha} \, \Phi^a \in \QH^*(X,\omega_{q}). $$
\end{defn}

Using degeneration arguments, Seidel \cite{seidel97} (in monotone case) and McDuff \cite{McDuff_seidel} proved that if $\rho_1,\rho_2$ are two commuting $\cpx^*$-actions and $\rho_3 = \rho_1 * \rho_2$ is the composition, then the corresponding Seidel elements $S_1,S_2,S_3$ satisfy the relation\footnote{Indeed they proved this in a more general situation in which the Seidel elements are generated by loops in $\pi_1 (\mathrm{Ham}(X,\omega))$.  Then every element in $\pi_1 (\mathrm{Ham}(X,\omega))$ gives a Seidel element which acts on $\QH^*(X)$ by quantum multiplication, and they showed that it defines an action of $\pi_1 (\mathrm{Ham}(X,\omega))$ on $\QH^*(X)$.}
$$ S_3 = S_1 * S_2,$$
(here $*$ denotes the quantum multiplication) under the following assignment of relations between Novikov variables of $E_1,E_2,E_3$.

The $\cpx^*$-actions $\rho_1,\rho_2$ define an $X$-bundle $\hat{E}$ over $\proj^1 \times \proj^1$:
\begin{equation} \label{eq:degfam}
\hat{E} := \big(X \times (\cpx^2 \setminus \{0\}) \times (\cpx^2 \setminus \{0\}) \big) / (\cpx^* \times \cpx^*)
\end{equation}
where $(\zeta_1,\zeta_2) \in \cpx^* \times \cpx^*$ acts by
$$(\zeta_1,\zeta_2) \cdot (x,(u_1,v_1),(u_2,v_2)) := (\rho_1(\zeta_1,\rho_2(\zeta_2,x)), (\zeta_1 u_1,\zeta_1 v_1), (\zeta_2 u_2,\zeta_2 v_2)).$$
$\hat{E}$ restricted to $\proj^1 \times \{[0,1]\}$ is the Seidel space $E_1$ associated to $\rho_1$; $\hat{E}$ restricted to $\{[1,0]\} \times \proj^1$ is the Seidel space $E_2$ associated to $\rho_2$; $\hat{E}$ restricted to the diagonal $\{(p_1,p_2) \in \proj^1 \times \proj^1: p_1 = p_2 \}$ is the Seidel space $E_3$ associated to the composition $\rho_3$.

Effective curves classes in $\hat{E}$ are generated by those of $E_1$ and $E_2$.  In particular, $\sigma_3$, the section class of $E_3$ which corresponds to the fixed locus in $X$ with negative weight under the action of $\rho_3$, when pushed forward to a curve class in $\hat{E}$, is of the form $d_1 + d_2$ for some $d_1 \in H_2^{\eff}(E_1)$ and $d_2 \in H_2^{\eff}(E_2)$.  Then assign the relation
$$q^{\sigma_3} := q^{d_1} q^{d_2}$$
between the Novikov variables $q^{d_1}$, $q^{d_2}$, $q^{\sigma_3}$ of $E_1,E_2$ and $E_3$ respectively.  The Novikov variables $q^\alpha$ of $E_i$, where $\alpha \in H_2^{\eff}(X)$, are set to equal for $i=1,2,3$.  Thus once the Novikov variables of $E_1$ and $E_2$ are fixed (with the requirement that $q^\alpha$ are the same for all $\alpha \in H_2^{\eff}(X)$), those of $E_3$ are automatically fixed.  In other words, $H_2^{\sec}(E_3)$ can be identified as $(H_2^{\sec}(E_1) \oplus H_2^{\sec}(E_2))/H_2(X)$, where $H_2(X)$ is embedded into $H_2^{\sec}(E_1) \oplus H_2^{\sec}(E_2)$ by $\alpha \mapsto (\alpha,-\alpha)$.

Returning to our case that $\rho_1$ is the action generated by $v_j$ and $\rho_2$ is the action generated by $-v_j$, their composition is the trivial action. We write $S_j$ and $S_j^-$ for the two Seidel elements. The Seidel element generated by the trivial action is simply $q^{\sigma_0}$, where $\sigma_0$ is a section of the trivial bundle $X \times \proj^1 \to \proj^1$.  The above assignment of Novikov variables gives
$$q^{\sigma_0} = q^{\sigma_{j} + \sigma_{j}^-}.$$
\begin{prop} \label{Cor Seidel}
One has the following equality for Seidel elements:
$$ S_j * S_j^- = q^{\sigma_{j} + \sigma_{j}^-}, $$
where the Novikov variables are regarded as elements in the (completed) group ring of $(H_2^{\sec}(E_1) \oplus H_2^{\sec}(E_2))/H_2(X)$.
\end{prop}

For the purpose of computing open GW invariants, we need the following definition:
\begin{defn}\label{defn:reg}
Let $Y$ be a K\"ahler manifold and $C_1,\ldots,C_k$ be complex analytic cycles in $Y$.  Denote by $\moduli_{0,k,d}^Y(C_1,\ldots,C_k)$ the fiber product of the (compactified) moduli space $\moduli^Y_{0,k,d}$ of stable maps in class $d$ with the cycles $C_1,\ldots,C_k$ by the evaluation maps.  Given an effective curve class $d_0$, let $\moduli^{Y, d_0 \, \reg}_{0,k,d}(C_1,\ldots,C_k)$ denote the union of those connected components of $\moduli_{0,k,d}^Y(C_1,\ldots,C_k)$ which contain a stable map with a holomorphic sphere component representing $d_0$. (Note that it can simply be an empty set, for instance, when $d - d_0$ is not effective.) Then define $\langle C_1, \ldots, C_k \rangle^{Y, d_0 \, \reg}_{0,k,d}$ to be the integration of $1$ over the virtual fundamental class associated to $\moduli^{Y, d_0 \, \reg}_{0,k,d}(C_1,\ldots,C_k)$.
\end{defn}

Intuitively $\langle C_1, \ldots, C_k \rangle^{Y, d_0 \, \reg}_{0,k,d}$ counts those genus-zero stable maps in the class $d$ passing through $C_1, \ldots, C_k$ which have a sphere component in the class $d_0$.
Notice that the above definition of $\langle C_1, \ldots, C_k \rangle^{Y, d_0 \, \reg}_{0,k,d}$ depends on the actual cycles rather than just the homology classes of $C_1, \ldots, C_k$.

From now on we denote by $D^{E^-}$ the push-forward of the toric prime divisor $D \subset X$ to the fiber $\cD_0$ of $E^- \to \bP^1$ at $0 \in \bP^1$.
The fiber of $E^- \to \bP^1$ at $\infty \in \bP^1$ is denoted as $\cD_\infty$.  We will apply the above definition to the moduli space $\moduli^{E^-,\sigma^- \, \reg}_{0,2,\sigma^-+\alpha}(D^{E^-}, \pt)$ where $\alpha \in H_2^{\eff,c_1=0}(X)$.  It is the connected component of $\moduli^{E^-}_{0,2,\sigma^-+\alpha}(D^{E^-}, \pt)$ which contains a rational curve with one holomorphic sphere component representing $\sigma^-$.  Then we have the
\begin{defn}[$\sigma^-$-regular GW invariants]\label{def:regGW}
$\langle D^{{E^-}}, [\pt]_{{E^-}} \rangle^{{E^-,\sigma^- \, \reg}}_{0,2,\sigma^-+\alpha}$ is defined as integration of $1$ over the virtual fundamental class of $\moduli^{E^-,\sigma^- \, \reg}_{0,2,\sigma^-+\alpha}(D^{{E^-}}, \pt)$.
\end{defn}

The following lemma, which will be useful later, says that every curve in $\moduli^{E^-,\sigma^- \, \reg}_{0,2,\sigma^-+\alpha}(D^{E^-}, \pt)$, where $\alpha \in H_2^{\eff,c_1=0}(X)$, has precisely one holomorphic sphere component representing $\sigma^-$:
\begin{lem} \label{lem:reg=>unique}
Let $\pt$ be a point in the open toric orbit of $E^-$.  Every rational curve in the moduli space $\moduli^{E^-,\sigma^- \, \reg}_{0,2,\sigma^-+\alpha}(D^{{E^-}}, \pt)$, where $\alpha \in H_2^{\eff,c_1=0}(X)$, consists of the unique holomorphic sphere component representing $\sigma^-$ passing through $\pt$ and $\cD_0 \cap \cD_j$, and some other components supported in $\cD_0$ representing $\alpha$.
\end{lem}

\begin{proof}
$\moduli^{E^-,\sigma^- \, \reg}_{0,2,\sigma^-+\alpha}(D^{{E^-}}, \pt)$ contains a rational curve which consists of a holomorphic sphere $C$ representing $\sigma^-$ and some other components representing $\alpha$.  Since $c_1(\alpha)=0$, the components representing $\alpha$ never pass through the generic point $\pt$ in the open toric orbit.  Thus $C$ has to pass through $\pt$.  Moreover either the holomorphic sphere $C$ passes through $D^{{E^-}} \subset \cD_0$, or the components representing $\alpha$ pass through $D^{{E^-}}$, which implies these components representing $\alpha$ are contained in $\cup_{l=1}^m \cD_0 \cap \cD_l$ and hence $C$ intersects $\cup_{l=1}^m \cD_0 \cap \cD_l$ (so that the whole curve is connected).  In both cases $C$ intersects $\cup_{l=1}^m \cD_0 \cap \cD_l$, implying that it intersects $\cD_0 \cap \cD_j$ since it represents $\sigma^-$.

Such a curve $C$ in class $\sigma^-$ passing through both $\pt$ and $\cD_0 \cap \cD_j$ is unique and not deformable.  Moreover, the nodal intersection between $C$ and $C'$ is not smoothable because of the following.  Suppose we can smooth out the nodal intersection.  Then we obtain a holomorphic sphere $\tilde{C}$ which passes through $\pt$ in the open toric orbit and represents $\sigma^- +\alpha'$, where $\alpha' \not= 0 \in H_2^{\eff, c_1=0}(X)$ since $X$ is semi-Fano.  The class $[\tilde{C}]$ is a non-negative linear combination of the basic disc classes $\beta_p$'s.  Now $c_1(\tilde{C}) = 3$, and $\tilde{C} \cdot \cD_0 = \tilde{C} \cdot \cD_\infty = 1$ (because $\alpha' \cdot \cD_0 = \alpha \cdot \cD_\infty = 0$).  This forces $\tilde{C} \cdot \cD_j = 1$ and $\tilde{C} \cdot \cD_p = 0$ for all $p \not= 0, \infty, j$.  Thus $\tilde{C}$ lies in the class $\sigma^-$, and so $\alpha' = 0$, a contradiction.

Thus if we consider another curve in the connected component $\moduli^{E^-,\sigma^- \, \reg}_{0,2,\sigma^-+\alpha}(D^{{E^-}}, \pt)$ which comes from a deformation of the curve $C$, it must have the same sphere component $C$.  Thus a rational curve in the moduli consists of $C$ union with a rational curve $C'$ representing $\alpha$.  Since $c_1(\alpha)=0$ and $\alpha$ is a fiber class, $C'$ must be supported in $\cup_{l=1}^m \cD_l$, see Lemma \ref{Lem alpha} below.  The sphere $C$ intersects $\cup_{l=1}^m \cD_l$ at exactly one point in $\cD_0$.  By connectedness of the rational curve $C'$ must be supported in $\cD_0$.
\end{proof}

\section{Relating open and closed invariants} \label{Sect open_closed}

Open GW invariants are difficult to compute in general because there are highly nontrivial obstructions to the moduli problems and, in contrast to closed GW theory, localization and degeneration formulas cannot be applied. In \cite{Chan10,LLW10}, under some strong restrictions on the geometry of the toric manifold $X$, it was shown that open GW invariants could be equated with certain closed GW invariants of $X$ (or certain toric compactifications of $X$ when $X$ is non-compact).  This gives an effective way to compute open GW invariants because closed GW invariants can be computed by various techniques.

However, for an arbitrary toric manifold $X$, the geometric technique in \cite{Chan10,LLW10} fails, and searching for spaces whose closed GW invariants correspond to open GW invariants of $X$ becomes much more difficult. An exciting discovery in this paper is that Seidel spaces associated to $X$, which are one dimensional higher than $X$, are indeed what we need in order to have such an open-closed comparison. Moreover it works for \emph{all} semi-Fano toric manifolds:

\begin{thm} \label{Thm open-closed}
Let $X$ be a semi-Fano toric manifold and $\beta \in \pi_2(X,\tor^X)$ a disc class of Maslov index 2 bounded by a Lagrangian torus fiber $\tor^X \subset X$.  Then $\beta$ must be of the form $\beta_j + \alpha$ for some basic disc class $\beta_j$ of $X$ ($j=1,\ldots,m$) and $\alpha \in H^{\eff}_2(X)$ with $c_1(\alpha) = 0$.

Let $v_j = \partial\beta_j \in N$ be the minimal generator of the corresponding ray in the fan of $X$. Let ${E_j^-}$ be the Seidel space corresponding to the $\cpx^*$-action generated by $-v_j = -\partial \beta_j$, and denote by $\tor^{{E_j^-}}$ a Lagrangian torus fiber of ${E_j^-}$.  Any class $a \in H^*(X)$ can be pushed forward (via Poincar\'e duality) by the inclusion $X \hookrightarrow {E_j^-}$ of $X$ as a fiber to give a class in $H^*(E_j^-)$, and it is denoted as $a^{E_j^-}$.

Let $v_i$ be a minimal generator and denote the corresponding toric prime divisor by $D=D_i$.  When $v_i \not\in F(v_j)$ or $D_l \cdot \alpha \not= 0$ for some $v_l \not\in F(v_j)$, where $F(v_j)$ is the minimal face of the fan polytope containing $v_j$, $n^X_{1,1}(\beta; D, [\pt]_{\tor^X})=0$.  Otherwise
$$ n^X_{1,1}(\beta; D, [\pt]_{\tor^X}) = \langle D^{{E_j^-}}, [\pt]_{{E_j^-}} \rangle^{E_j^-,\,\sigma_j^-\,\reg}_{0,2,\sigma_j^-+\alpha^{E_j^-}}$$
where $\sigma_j^-\in H_2({E_j^-})$ is the zero section class of ${E_j^-}$ (see Definition \ref{defn section}), $[\pt]_{\tor^X} \in H^n(\tor^X)$ is a point class of $\tor^X$ and $[\pt]_{{E_j^-}} \in H^{2n}({E_j^-})$ is a point class of ${E_j^-}$.  The $\sigma_j^-$-regular Gromov-Witten invariant on the right-hand side is defined in Definition \ref{def:regGW}.
\end{thm}

\begin{rmk}
\hfill
\begin{enumerate}
\item
As suggested by a referee, the equality in the above theorem should hold true for all $i$, $j$ and $\alpha$ because the regular GW invariant in the right hand side also vanishes if $v_i \notin F(v_j)$ or $D_l \cdot \alpha \neq 0$ for some $v_l \notin F(v_j)$, but the current statement suffices for the purposes of this paper.

\item
In this paper we consider open GW invariants defined using Kuranishi structures. However we would like to point out that the above formula in Theorem \ref{Thm open-closed} remains valid whenever reasonable analytic structures are put on the moduli spaces to define GW invariants. This is because the way we compare moduli spaces of stable discs and maps, as detailed in the proofs of Propositions \ref{prop open mod compare} and \ref{Prop open-closed mod}, is geometric in nature and it identifies the deformation and obstruction theories of the two moduli problems on the nose.
\end{enumerate}
\end{rmk}

The statement that a stable disc class of Maslov index 2 bounded by $\tor^X$ is of the form $\beta_j + \alpha$ was proved by Cho-Oh \cite{cho06} and Fukaya-Oh-Ohta-Ono \cite{FOOO1}, and it is recalled in Lemma \ref{Lem Maslov-two stable disc}. We also need the following lemma about curves in the Seidel space ${E_j^-}$ representing $\alpha$:

\begin{lem} \label{Lem alpha}
Assume the setting as in Theorem \ref{Thm open-closed}.  Let $C \subset {E_j^-}$ be a rational curve representing a fiber class (i.e. a class in $H_2^{\eff}(X)$) of ${E_j^-} \to \proj^1$ with $c_1(C) \leq 1$.  Then $C \subset \bigcup_{l=1}^m \mathscr{D}_l$.
\end{lem}
\begin{proof}
Since $C$ represents a fiber class, its image under ${E_j^-} \to \proj^1$ can only be a point, which means $C$ belongs to a fiber of ${E_j^-} \to \proj^1$, which is identified as $X$.  Note that the Chern number of $C$ in $X$ is the same as that of $C$ in $E_j^-$. Since $X$ is semi-Fano, every component of $C$ has non-negative Chern number.  But $c_1(C) \leq 1$.  Thus each component of $C$ has $c_1 \leq 1$. Let $C'\subset C$ be a component. Then $-K_X\cdot C'$ is either $0$ or $1$. Suppose $-K_X\cdot C'=0$. It is impossible to have $D_i\cdot C'=0$ for all $i$ since this means $[C']=0$. So there exists an $i$ such that $D_i\cdot C'<0$, which implies that $C'\subset D_i$. Suppose $-K_X\cdot C'=1$. It is possible that $D_i\cdot C'<0$ for some $i$, which implies $C'\subset D_i$. The other possibility is that $D_i\cdot C'=1$ for some $i$ and $D_j\cdot C'=0$ for $j\neq i$. This is impossible since it violates linear relations. We thus conclude that every component of $C$ lies in a toric divisor of $X$.  Under the inclusion $X \hookrightarrow {E_j^-}$ as a fiber, $D_l \subset \mathscr{D}_l$ for all $l = 1, \ldots, m$.  Thus $C \subset \bigcup_{l=1}^m \mathscr{D}_l$.
\end{proof}

Now consider the easier case $v_i \not\in F(v_j)$ or $D_l \cdot \alpha \neq 0$ for some $v_l \not\in F(v_j)$ of Theorem \ref{Thm open-closed}.  We will use the following lemma.

\begin{lem}[Lemma 4.5 of \cite{G-I11}] \label{lem:G-I}
Let $\sigma$ be a cone in $\Sigma$.  Suppose that $d \in H_2(X)$ satisfies $c_1(d)=0$ and $D_i \cdot d \geq 0$ for all $i$ such that $v_i \not\in \sigma$.  Then $d$ is effective and $D_i \cdot d = 0$ for all $i$ such that $v_i \not\in F(\sigma)$, where $F(\sigma)$ denotes the minimal face of the fan polytope containing the primitive generators in $\sigma$.
\end{lem}

The following consequence will be useful later.

\begin{cor} \label{cor:g}
$\exp(g_j(\check{q}(q)))$ only involves Novikov variables $q^d$ with $c_1(d)=0$ and $D_i \cdot d = 0$ whenever $v_i \not\in F(v_j)$.
\end{cor}
\begin{proof}
By definition $g_j(\check{q})$ is a summation over curve classes $d$ with $c_1(d)=0$ and $D_i \cdot d\geq 0$ for all $i \neq j$.  By Lemma \ref{lem:G-I}, $D_i \cdot d = 0$ whenever $v_i \not\in F(v_j)$.  Hence $g_j(\check{q})$ involves $\check{q}^d$ where $D_i \cdot d = 0$ whenever $v_i \not\in F(v_j)$.  For such $d$, the mirror map $\log q^d = \log \check{q}^d - \sum_{v_i \in F(v_j)} (D_i \cdot d) g_i(\check{q})$ also involves only $\check{q}^{d'}$ with $D_l \cdot d' = 0$ whenever $v_l \not\in F(v_j)$ (because $F(v_i) \subset F(v_j)$ if $v_i \in F(v_j)$).  Such $d$'s satisfying $c_1(d)=0$ and $D_i \cdot d = 0$ whenever $v_i \not\in F(v_j)$ form a subcone of the Mori cone.  Then the inverse mirror map $\check{q}^{d}(q)$ only depends on $q^d$ with $c_1(d)=0$ and $D_i \cdot d = 0$ whenever $v_i \not\in F(v_j)$.
\end{proof}

\begin{prop} \label{prop:Cin}
A connected rational curve $C$ in $X$ with $c_1(C)=0$ which has a sphere component intersecting the open toric orbit of $D_j$ (as a toric manifold itself) must be contained in $\bigcup_{i: v_i \in F(v_j)} D_i$, and $D_i \cdot [C] = 0$ whenever $v_i \not\in F(v_j)$.
\end{prop}

\begin{proof}
All sphere components of $C$ lie in toric divisors of $X$ since $c_1(C)=0$.  Let $C_1$ be a holomorphic sphere component of $C$ lying in $D_j$ which intersects the open toric orbit of $D_j$.  It satisfies $D_j \cdot C_1 < 0$ and $D_i \cdot C_1 \geq 0$ for all $i\not=j$.  By Lemma \ref{lem:G-I} applied to the cone  $\R_{\geq 0} v_j$, we have $D_i \cdot C_1 = 0$, and so $C_1 \cap D_i = \emptyset$, for all $v_i \not\in F(v_j)$.

Now consider another sphere component $C_2$ of $C$ contained in some $D_{j_2}$ which intersects $C_1$ at a nodal point $p$ lying in $D_{j_2} \cap D_j$.  Then $v_{j_2} \in F(v_j)$.  Consider the minimal toric strata containing $p$, which is dual to a certain cone $\sigma$ in the fan containing $v_{j_2}$ and $v_j$.  Since $p$ does not lie in $D_i$ for any $v_i \not\in F(v_j)$, $\sigma$ is contained in $F(v_j)$.  Consider a toric prime divisor $D$ with $D \cdot C_2 < 0$.  Then $p \in C_2 \subset D$, and hence the minimal toric strata containing $p$ is a subset of $D$.  Thus $D$ must correspond to a primitive generator in $\sigma$.  This proves $D_i \cdot C_2 \geq 0$ for all $v_i \not\in \sigma$.  By Lemma \ref{lem:G-I} applied to the cone $\sigma$, we have $D_i \cdot C_2 = 0$ for all $v_i \not\in F(\sigma) = F(v_j)$.  Inductively all sphere components of $C$ are contained in $\bigcup_{i: v_i \in F(v_j)} D_i$.
\end{proof}

Since $n_{1,1}^X (\beta_j + \alpha; D_i, [\pt]_L) = (D_i \cdot (\beta_j+\alpha)) n_1(\beta_j + \alpha)$ (Theorem \ref{thm div eq open}), we obtain
\begin{cor} \label{cor:n}
$n_{1,1}^X (\beta_j + \alpha; D_i, [\pt]_L) = 0$ if $v_i \not\in F(v_j)$ or $D_l \cdot \alpha \neq 0$ for some $v_l \not\in F(v_j)$.  Moreover the generating function $\sum_\alpha q^\alpha n_1(\beta_j+\alpha)$ has only Novikov variables $q^\alpha$ with $D_i \cdot \alpha = 0$ whenever $v_i \not\in F(v_j)$.
\end{cor}
\begin{proof}
Let $\beta_j+\alpha$ be represented by a union of basic disc $D$ representing $\beta_j$ and a rational curve $C$ representing $\alpha$, where $D$ and $C$ intersect at a node.  Then $C$ has a sphere component intersecting the open toric orbit of $D_j$, and hence by Proposition \ref{prop:Cin} $D_i \cdot \alpha = 0$ for all $v_i \not\in F(v_j)$.  So $n_1(\beta_j+\alpha) \not=0$ only when $D_i \cdot \alpha = 0$ for all $v_i \not\in F(v_j)$.  Moreover $n_{1,1}^X (\beta_j + \alpha; D_i, [\pt]_L) = (D_i \cdot (\beta_j+\alpha)) n_1(\beta_j + \alpha) = 0$ if $v_i \not\in F(v_j)$ or $D_l \cdot \alpha \neq 0$ for some $v_l \not\in F(v_j)$.
\end{proof}

The above proves Theorem \ref{Thm open-closed} in the case $v_i \not\in F(v_j)$. The rest of this section is devoted to proving Theorem \ref{Thm open-closed} in the case $v_i \in F(v_j)$ and $D_l \cdot \alpha = 0$ for all $v_l \not\in F(v_j)$.  The proof is divided into two main steps.  First, we equate the open GW invariant $n_{1,1}(\beta; D, [\pt]_{\tor^X})$ of $X$ to a certain open GW invariant of ${E_j^-}$ (Theorem \ref{Thm X-{E_j^-}-open}).  Then we show that this open GW invariant of $E_j^-$ is equal to the closed GW invariant
$\langle D^{E_j^-}, [\pt]_{{E_j^-}} \rangle^{E_j^-,\,\sigma_j^-\,\reg}_{0,2,\sigma_j^-+\alpha}$
of ${E_j^-}$ (Theorem \ref{Thm {E_j^-}-open-closed}). Here $D^{E_j^-}\in H^4(E_j^-)$ is the push-forward of $D\in H^2(X)$ under the inclusion $X\hookrightarrow E_j^-$ of $X$ as a fiber. Since $D$ is a divisor of $X$, $D^{{E_j^-}}$ is of complex codimension 2 in ${E_j^-}$.

\subsection{First step}
The precise statement of the first main step is the following:
\begin{thm} \label{Thm X-{E_j^-}-open}
Assume the notations as in Theorem \ref{Thm open-closed}.  Then
$$ n^X_{1,1}(\beta_j+\alpha; D, [\pt]_{\tor^X}) = n^{{E_j^-}}_{1,1}(b_0 + b_j + \alpha; D^{E_j^-}, [\pt]_{\tor^{E_j^-}}) $$
where we recall that $b_l$ ($l = 0, 1, \ldots, m, \infty$) are the basic disc classes of ${E_j^-}$ (see Section \ref{Sect Seidel_rep}).  Moreover $[\pt]_{\tor^{{E_j^-}}} \in H^{n+1} (\tor^{{E_j^-}})$ denotes the point class of the Lagrangian torus fiber of ${E_j^-}$.
\end{thm}

Recall that $ n^X_{1,1}(\beta_j+\alpha; D, [\pt]_{\tor^X}) = \pairing{[\moduli^{\mathrm{op}}_{1,1} (\beta_j+\alpha; D)]_{\virt}}{[\pt]_{\tor^X}} \in \rat$,
and by definition of Poincar\'e pairing, this is the same as
$$\iota_{\tor^{X}}^* [\moduli^{\mathrm{op}}_{1,1}(\beta_j+\alpha; D)]_{\mathrm{virt}} \in H^0(\pt, \rat) \cong \rat.$$
where $\iota_{\tor^{X}}: \{\pt\} \hookrightarrow \tor^X$ is an inclusion of a point to $\tor^X$.  Similarly
$$n^{E_j^-}_{1,1}(b_0 + b_j + \alpha; D^{{E_j^-}}, [\pt]_{\tor^{{E_j^-}}}) = \iota_{\tor^{{E_j^-}}}^* [\moduli^{\mathrm{op}}_{1,1}(b_0+b_j+\alpha; D^{{E_j^-}})]_{\mathrm{virt}} \in H^0(\pt, \rat) \cong \rat$$
where $\iota_{\tor^{{E_j^-}}}: \{\pt\} \hookrightarrow \tor^{E_j^-}$ is an inclusion of a point to $\tor^{E_j^-}$.  We denote the images of $\iota_{\tor^{X}}$ and $\iota_{\tor^{{E_j^-}}}$ to be $\pt_{\tor^X}$ and $\pt_{\tor^{E_j^-}}$ respectively.

By Lemma \ref{Lem Maslov-two stable disc}, a stable disc in $\moduli^{\mathrm{op}}_{1,1}(\beta; D)$ has only one disc component.  Thus it never splits into the union of two stable discs.  Hence $\moduli^{\mathrm{op}}_{1,1}(\beta; D)$ has no codimension one boundary.  The following key lemma shows that $\moduli^{\mathrm{op}}_{1,1}(b_0+b_j+\alpha; D^{E_j^-})$ also has this property, whose proof requires a more careful analysis of the stable discs since $b_0+b_j$ has Maslov index 4 (which is not the minimal Maslov index of $\tor^{E_j^-}$) and ${E_j^-}$ may not be semi-Fano:

\begin{lem} \label{Lem Maslov 4}
Assume the above settings.  A stable disc in $\moduli^{\mathrm{op}}_{1,1}(b_0+b_j+\alpha; D^{E_j^-})$ consists of a holomorphic disc component and a rational curve, which meet at only one nodal point.  The disc component belongs to the class $b_0 + b_j$ for some $j=1, \ldots, m$, and the rational curve belongs to $\alpha$.  In particular, $\moduli^{\mathrm{op}}_{1,1}(b_0+b_j+\alpha; D^{E_j^-})$ has no codimension one boundary.
\end{lem}
\begin{proof}
Consider a stable disc $\phi$ in $\moduli^{\mathrm{op}}_{1,1}(b_0+b_j+\alpha; D^{E_j^-})$.  It consists of several disc components and sphere components.  Notice that $\pairing{b_0+b_j+\alpha}{\mathscr{D}_\infty}=0$, where $\pairing{\cdot}{\cdot}$ denotes the pairing between $H_2(E_j^-, \tor^{E_j^-})$ and $H^2(E_j^-, \tor^{E_j^-})$.  Since every holomorphic disc bounded by $\tor^{E_j^-}$ and every holomorphic sphere in ${E_j^-}$ has non-negative intersection with $\mathscr{D}_\infty$, this implies that each sphere component of $\phi$ has intersection number $0$ with $\mathscr{D}_\infty$.  So every sphere component of $\phi$ is in a fiber class as otherwise it would have positive intersection number with $\mathscr{D}_\infty$.  In particular each sphere component of $\phi$ has non-negative Chern number and is contained in a fiber of ${E_j^-} \to \proj^1$.  Together with the fact that $\phi$ has Maslov index $\mu(b_0)+\mu(b_j)+2 c_1(\alpha)=4$, this implies that each disc component has at most Maslov index 4.

Suppose a disc component of $\phi$ has Maslov index 4.  Then all the sphere components have Chern number zero.  Since every non-constant holomorphic disc has Maslov index at least 2, the other disc components of $\phi$ must be constant, and they are mapped to $\tor^{E_j^-}$.  On the other hand the interior marked point $p^{\mathrm{int}}$ of $\phi$ has to be mapped to $D^{E_j^-}$, which sits inside the fiber $D_0$ and is disjoint from $\tor^{E_j^-}$.  Hence $p^{\mathrm{int}}$ cannot be located in the constant disc components. But then at least one of the constant disc components does not have 3 special points, making $\phi$ unstable. This shows that $\phi$ has only one disc component which has Maslov index 4.

Then we prove that the disc component is attached with the holomorphic spheres at only one interior nodal point.  Holomorphic discs bounded by a Lagrangian torus fiber have been classified by Cho-Oh \cite{cho06}.  In particular if a holomorphic disc of Maslov index 4 passes through $D_l^{E_j^-}$ for any $l$, it intersects with the union of toric divisors at only one single interior point.  On the other hand, by Lemma \ref{Lem alpha}, all the sphere components are mapped to the union of the toric divisors $\mathscr{D}_l$, $l = 1, \ldots, m$.  Thus the disc component must passes through one $D_l^{E_j^-}$ and is attached with the holomorphic spheres at only one interior nodal point.  This implies that $\phi$ is the union of a holomorphic disc and a rational curve joint at a single nodal point.  The disc component belongs to $b_0 + b_l$  for some $l$ and the rational curve component belongs to a certain class $\rho$.  Then $b_0 + b_l + \rho = b_0 + b_j + \alpha$ as disc classes, which forces $l = j$ and $\rho = \alpha$.  Hence the holomorphic disc represents $b_0 + b_j$, and the rational curve must represent $\alpha$.

Now suppose otherwise that every disc component of $\phi$ has Maslov index less than 4.  Then $\phi$ must have a disc component of Maslov index 2.  Then the other disc components have Maslov index at most two, and the sphere components have Chern number at most one.  Moreover the sphere components belong to some fiber classes.  By Lemma \ref{Lem alpha}, each of them is contained in $\mathscr{D}_l$ for some $l = 1, \ldots, m$.

A holomorphic disc of Maslov index at most two does not pass through $D^{E_j^-}$.  Thus the interior marked point $p^{\mathrm{int}}$ of $\phi$ must be located in a sphere component.  But $\phi(p^{\mathrm{int}}) \in D^{E_j^-}$ which is contained in the fiber at $0$.  This implies that this sphere component is contained in $\mathscr{D}_0 \cap \mathscr{D}_l$ for some $l = 1, \ldots, m$.  However, a holomorphic disc of Maslov index at most two does not pass through $\mathscr{D}_0 \cap \mathscr{D}_l$, and so none of the disc components is connected to this sphere component.  We thus conclude that this situation cannot occur.

We have now proved that $\phi$ has only one disc component.  This implies that it never splits into two stable discs, meaning that disc bubbling never occurs.  Thus the moduli space $\moduli^{\mathrm{op}}_{1,1}(b_0+b_j+\alpha; D^{E_j^-})$ has no codimension one boundary.
\end{proof}

Now both $\moduli^{\mathrm{op}}_{1,1}(\beta_j+\alpha; D)$ and $\moduli^{\mathrm{op}}_{1,1}(b_0+b_j+\alpha; D^{E_j^-})$ have no codimension one boundaries.  By \cite[Lemma A1.43]{FOOO_I}, we have
\begin{align*}
n^X_{1,1}(\beta_j+\alpha; D, [\pt]_{\tor^X}) &= \iota_{\tor^{X}}^* [\moduli^{\mathrm{op}}_{1,1}(\beta_j+\alpha; D)]_{\mathrm{virt}} \\
&= [\moduli^{\mathrm{op}}_{1,1}(\beta_j+\alpha; D, \pt_{\tor^X})]_{\mathrm{virt}} \in H^{\mathrm{top}}(D \times \{\pt_{\tor^X}\}, \rat) = \rat,
\end{align*}
and
\begin{align}
n^{{E_j^-}}_{1,1}(b_0 + b_j + \alpha; D^{{E_j^-}}, [\pt]_{\tor^{{E_j^-}}}) &= \iota_{\tor^{{E_j^-}}}^* [\moduli^{\mathrm{op}}_{1,1}(b_0+b_j+\alpha; D^{E_j^-})]_{\mathrm{virt}} \notag \\
\label{eqn open_E} &= [\moduli^{\mathrm{op}}_{1,1}(b_0+b_j+\alpha; D^{E_j^-}, \pt_{\tor^{E_j^-}})]_{\mathrm{virt}} \\
&\in H^\mathrm{top}(D^{E_j^-} \times \{\pt_{\tor^{E_j^-}}\}, \rat) = \rat \notag,
\end{align}
where
\begin{align*}
\moduli^{\mathrm{op}}_{1,1}(\beta_j+\alpha; D_i, \pt_{\tor^X}) &= (\moduli^{\mathrm{op}}_{1,1} (\beta_j + \alpha) \times_{X} D_i) \times_{\tor^X} \{\pt_{\tor^X}\}\\
&= \moduli^{\mathrm{op}}_{1,1} (\beta_j + \alpha) \times_{X \times \tor^X} (D_i \times \{\pt_{\tor^X}\})
\end{align*}
and
\begin{align*}
\moduli^{\mathrm{op}}_{1,1}(b_0+b_j+\alpha; D_i^{E_j^-}, \pt_{\tor^{E_j^-}}) &= (\moduli^{\mathrm{op}}_{1,1} (b_0+b_j + \alpha) \times_{{E_j^-}} D_i^{E_j^-}) \times_{\tor^{E_j^-}} \{\pt_{\tor^{E_j^-}}\}\\
&= \moduli^{\mathrm{op}}_{1,1} (b_0+b_j + \alpha) \times_{{E_j^-} \times \tor^{E_j^-}} (D_i^{E_j^-} \times \{\pt_{\tor^{E_j^-}}\}).
\end{align*}
The fiber products appeared above use the evaluation maps $\mathrm{ev}^X_+: \moduli^{\mathrm{op}}_{1,1} (\beta_j + \alpha) \to X$, $\mathrm{ev}^X_0: \moduli^{\mathrm{op}}_{1,1} (\beta_j + \alpha) \to \tor^X$, $\mathrm{ev}^E_+: \moduli^{\mathrm{op}}_{1,1} (\beta_j + \alpha) \to {E_j^-}$, $\mathrm{ev}^E_0: \moduli^{\mathrm{op}}_{1,1} (\beta_j + \alpha) \to \tor^{E_j^-}$, and the inclusion maps $D_i  \hookrightarrow X$, $\{\pt_{\tor^X}\} \hookrightarrow \tor^X$, $D^E_i  \hookrightarrow {E_j^-}$, $\{\pt_{\tor^{E_j^-}}\} \hookrightarrow \tor^{E_j^-}$.

Thus, in order to prove
$n^X_{1,1}(\beta_j+\alpha; D, [\pt]_{\tor^X}) = n^{{E_j^-}}_{1,1}(b_0 + b_j + \alpha; D^{{E_j^-}}, [\pt]_{\tor^{{E_j^-}}}),$
it suffices to show 
the following
\begin{prop} \label{prop open mod compare}
Fix a point $\pt_{\tor^X} \in \tor^X$ and a point $\pt_{\tor^{E_j^-}} \in \tor^{E_j^-}$.  Then we have
\begin{equation}\label{open=open_K_spaces}
\moduli^{\mathrm{op}}_{1,1}(\beta_j+\alpha; D, \pt_{\tor^X}) \cong \moduli^{\mathrm{op}}_{1,1}(b_0+b_j+\alpha; D^{E_j^-}, \pt_{\tor^{E_j^-}})
\end{equation}
as Kuranishi spaces.
\end{prop}

\begin{proof}
We divide the proof into three parts.\\
\noindent{\bf (A) Virtual dimensions.}
First of all, both sides have virtual dimension zero:
$$ \dim \moduli^{\mathrm{op}}_{1,1} (\beta_j + \alpha) = \mu (\beta_j) + 2 c_1(\alpha) + 2 + 1 + n - 3 = 2 + n.$$
Requiring the interior marked point to pass through $D$ cuts down the dimension by $2$; requiring the boundary marked point to pass through $\pt_{\tor^X}$ further cuts down the dimension by $n$.  Thus the virtual dimension of the LHS of \eqref{open=open_K_spaces} is zero.  For the RHS of \eqref{open=open_K_spaces},
$$\dim  \moduli^{\mathrm{op}}_{1,1} (b_0 + b_j + \alpha) = \mu (b_0) + \mu(b_j) + 2 c_1(\alpha) + 2 + 1 + (n+1) - 3 = 5 + n.$$
Requiring the interior marked point to pass through $D^{E_j^-}$ cuts down the dimension by $4$; requiring the boundary marked point to pass through $\pt_{\tor^{E_j^-}}$ further cuts down the dimension by $n+1$.  Thus the virtual dimension of the RHS of \eqref{open=open_K_spaces} is also zero.

\noindent{\bf (B) Spaces.} In what follows the domain interior marked point of a stable disc is always denoted as $p^{\mathrm{int}}$, and the domain boundary marked point is always denoted as $p^{\mathrm{bdy}}$.

Now we construct a bijection between the left-hand side and the right-hand side of \eqref{open=open_K_spaces}.  In the following we fix a local toric chart $\chi = (\chi_1, \ldots, \chi_n)$ of $X$ which covers the open orbit of $D_j \subset X$, and such that $\chi_1(D_j) = 0$.  Without loss of generality, we may take $\tor^X$ to be the fiber $|\chi_l| = 1$ for all $l$, and $\pt_{\tor^X}$ to be $\chi_l(\pt_{\tor^X}) = 1$ for all $l$.  Correspondingly we have the local chart $(\chi,w)$ of ${E_j^-}$ around the fiber $w=0 \in \proj^1$.  Without loss of generality we take $\tor^{E_j^-}$ to be the fiber $|\chi_l| = |w| = 1$, and $\pt_{\tor^{E_j^-}}$ to be $\chi_l(\pt_{\tor^{E_j^-}}) = w(\pt_{\tor^{E_j^-}}) = 1$ for all $l$.

First consider the easier case $\alpha = 0$.  By Lemma \ref{Lem Maslov-two stable disc}, a stable disc in the LHS is a holomorphic disc $u$ in class $\beta_j$.  The domain is a closed unit disc $\Delta \subset \cpx$.  By using automorphism we may take $p^{\mathrm{int}} = 0$ and $p^{\mathrm{bdy}} = 1$.  In the above chosen local coordinates of $X$, $u$ has the expression
$$u(z) = \left(\conste^{\consti\theta_1} \frac{z-\alpha_0}{1+\bar{\alpha_0} z}, \conste^{\consti\theta_2}, \ldots, \conste^{\consti\theta_n}\right)$$
for some $\alpha_0 \in \Delta$ and $\theta_k \in \real$ for $k = 1,\ldots, n$.  $u$ passes through $D = D_i$ only when $i = j$.  Thus the left-hand side is simply an empty set when $i \neq j$.  When $i = j$, $u(0) \in D_j$ forces $\alpha_0 = 0$, and requiring $u(1) = \pt_{\tor_X}=(1,\ldots,1)$ fixes $\theta_k \cong 0$ for all $k = 1,\ldots,n$.  Thus the left-hand side is the empty set when $i \neq j$, and is a singleton when $i = j$.

On the other side, by Lemma \ref{Lem Maslov 4}, a stable disc in the RHS is a holomorphic disc $\nu$ in class $b_0 + b_j$.  Such discs are also classified by Cho-Oh \cite{cho06}.  Again we use the domain automorphism to fix $p^{\mathrm{bdy}} = 1$ and $p^{\mathrm{int}} = 0$.  Then the disc is of the form
$$w \circ \nu(z) = \conste^{\consti\theta_0} \frac{z-\alpha_1}{1+\bar{\alpha}_1z}, \quad \chi \circ \nu (z) = \left(\conste^{\consti\theta_1} \frac{z-\alpha_2}{1+\bar{\alpha}_2z},\conste^{\consti\theta_2}, \ldots, \conste^{\consti\theta_n}\right),$$
where $\alpha_1,\alpha_2 \in \Delta$ and $\theta_i \in \real$.  $\nu$ never hits $D_i^{E_j^-}$ when $i \neq j$.  When $i = j$, $\nu(0) \in D_j^{E_j^-}$ forces $w = \chi_1 = 0$ when $z = 0$.  Then $\alpha_1 = \alpha_2 = 0$.  Also $\nu(1) = \pt_{\tor^{E_j^-}}$ means $w = \chi_1 = \cdots = \chi_n = 1$ when $z = 1$, which implies $\theta_0 = \cdots = \theta_n = 1$.  Thus the moduli space in the RHS is empty when $i \neq j$, and is a singleton when $i = j$.  This verifies that the LHS matches with the RHS.

Now consider the case $\alpha \neq 0$.  Let $\phi$ be a stable disc bounded by $\tor^X$ in the LHS of \eqref{open=open_K_spaces}.  We associate $\phi$ with a stable disc bounded by $\tor^{E_j^-}$ in the RHS of \eqref{open=open_K_spaces} as follows.  By Lemma \ref{Lem Maslov-two stable disc}, $\phi$ is a holomorphic disc in class $\beta_j$ attached with a rational curve in class $\alpha$ at an interior nodal point.  Let us identify the domain disc component with the closed unit disc $\Delta \subset \cpx$, denote the domain of the rational curve by $C$, and denote $\phi_\Delta := \phi|_{\Delta}, \phi_C:= \phi|_{C}$.  The nodal point corresponds to a point $p^{\mathrm{nod}} \in C$ and a point in $\Delta$.  By using automorphism of $\Delta$ we may assume this point to be $0$ and  $p^{\mathrm{bdy}} = 1$.  Then $\phi_\Delta(0) = \phi_C(p^{\mathrm{nod}})$.  In the chosen local coordinates $\chi$, we have
$$\phi_\Delta(z) = \left(\conste^{\consti\theta_1} \frac{z-\alpha_0}{1+\bar{\alpha_0} z},\conste^{\consti\theta_2}, \ldots, \conste^{\consti\theta_n}\right)$$
for some $\alpha_0 \in \Delta, \theta_l \in \real$ for $l = 1, \ldots, n$.

Since $\phi_C$ has Chern number zero, $\phi(C) \subset \bigcup_l D_l$, and in particular $\phi_\Delta(0) = \phi_C (p^{\mathrm{nod}}) \in \bigcup_l D_l$.  But $\phi_\Delta$ does not hit any toric divisors except $D_j$.  Thus $\phi_\Delta(0) \in D_j$, and $0 \in \Delta$ is the only point which is mapped to $\bigcup_l D_l$ under $\phi_\Delta$.  This forces $\alpha_0 = 0$ in the above expression of $\phi_\Delta$.  Moreover $\phi_\Delta$ maps $z=1$ to $\pt_{\tor^X} = (1, \ldots, 1)$, and this forces $\theta_1 \cong \theta_2 \cong \cdots \cong \theta_n \cong 0$.  As a result, $\phi_\Delta = (z, 1, \ldots, 1)$.
On the other hand $\phi(p^{\mathrm{int}}) \in D_i$.  Suppose $p^{\mathrm{int}}$ lies on the disc component.  Since $p^{\mathrm{int}}$ has to be different from the nodal point, $p^{\mathrm{int}} \neq 0$.  But then $\phi_\Delta(p^{\mathrm{int}}) \not\in \bigcup_l D_l$, and so $p^{\mathrm{int}}$ is not mapped to $D_i$, a contradiction.  Thus $p^{\mathrm{int}}$ has to be located in the rational curve $C$.

We associate to $\phi$ an element $\phi^{E_j^-}$ in the RHS which has the same domain and marked points $p^{\mathrm{bdy}},p^{\mathrm{int}}$ as $\phi$ (the domain is $\Delta$ attached with $C$ at $z=0$).  $\phi^{E_j^-}|_\Delta$ is defined to be $(\phi_\Delta(z),z)$ written in terms of the above chosen local coordinates $(\chi,w)$ of ${E_j^-}$, and $\phi^{E_j^-}|_C := (\phi_C,0)$.  Notice that $\phi^{E_j^-}|_\Delta(0) = (\phi_\Delta(0),0) = (\phi_C(p^{\mathrm{nod}}),0) = \phi^{E_j^-}|_C(p^{\mathrm{nod}})$, and so $\phi^{E_j^-}$ is well-defined.  Moreover since $\phi_C(p^{\mathrm{int}}) \in D_i$, $\phi^{E_j^-}(p^{\mathrm{int}}) = (\phi_C(p^{\mathrm{int}}),0) \in D_i^{E_j^-}$.  Also $\phi^{E_j^-}(p^{\mathrm{bdy}}) = (\phi_\Delta(1),1) = \pt_{\tor^{E_j^-}}$.  This verifies that $\phi^{E_j^-}$ is an element in the RHS.

Now we prove that every element in the RHS of \eqref{open=open_K_spaces} comes from an element from the LHS of \eqref{open=open_K_spaces} in the way we described above.  By Lemma \ref{Lem Maslov 4}, a stable disc $\phi^{E_j^-}$ in $b_0 + b_j + \alpha$ must be a holomorphic disc representing $b_0 + b_j$ attached with a rational curve of Chern number zero representing $\alpha$.  As above, the domain disc component is identified with the unit disc $\Delta\subset\cpx$, and the domain rational curve is denoted by $C$.  The nodal point corresponds to a point $p^{\mathrm{nod}} \in C$ and a point in $\Delta$.  By using automorphism of $\Delta$ we may assume this point to be $0$ and  $p^{\mathrm{bdy}} = 1$.  Then $\phi^{E_j^-}_\Delta(0) = \phi^{E_j^-}_C(p^{\mathrm{nod}})$.  Using Cho-Oh's classification of holomorphic discs \cite{cho06}, in the chosen local coordinates $(w,\chi)$, $\phi^{E_j^-}_\Delta$ is of the form
$$w \circ \phi^{E_j^-}_\Delta(z) = \conste^{\consti\theta_0} \frac{z-\alpha_0}{1+\bar{\alpha}_0 z}, \quad
\chi \circ \phi^{E_j^-}_\Delta (z) = \left(\conste^{\consti\theta_1} \frac{z-\alpha_1}{1-\bar{\alpha_1}z}, \conste^{\consti\theta_2}, \ldots, \conste^{\consti\theta_n}\right).$$

Suppose $p^\mathrm{int}$ lies in the disc component.  Then $\phi^{E_j^-}_\Delta(p^\mathrm{int}) \in D_i^{E_j^-}$.  This happens only when $i=j$, $\alpha_0 = \alpha_1 = 0$.  In such case $\phi^{E_j^-}_\Delta$ hits the union of toric divisors of ${E_j^-}$ only at one point $z=p^{\mathrm{int}}$.  Now $\phi^{E_j^-}_C$ represents the fiber class $\alpha$ with $c_1(\alpha) = 0$, and so by Lemma \ref{Lem alpha} $\phi^{E_j^-}(C) \subset \bigcup_{l=1}^m \mathscr{D}_l$.  In particular $\phi^{E_j^-}_\Delta(0) = \phi^{E_j^-}_C(p^{\mathrm{nod}}) \in \bigcup_{l=1}^m \mathscr{D}_l$.  This forces $p^{\mathrm{int}}$ to coincide with the nodal point, a contradiction.  Thus $p^\mathrm{int}$ must lie in the rational curve $C$.

The image of $\phi^{E_j^-}_C$ lies in a fiber of ${E_j^-} \to \proj^1$.  But since $\phi^{E_j^-}_C(p^\mathrm{int}) \in D_i^{E_j^-}$ which lies in $\mathscr{D}_0$ (the fiber at zero), this forces $\phi^{E_j^-}_C$ to lie in $\mathscr{D}_0$.  Then $\phi^{E_j^-}_C$ is of the form $(0, \phi_C)$ in the local coordinates $(w,\chi)$.  Together with $\phi^{E_j^-}(C) \subset \bigcup_{l=1}^m \mathscr{D}_l$, this means $\phi^{E_j^-}(C) \subset \bigcup_{l=1}^m D^{E_j^-}_l$.  Then $\phi^{E_j^-}_\Delta(0) = \phi^{E_j^-}_C(p^{\mathrm{nod}}) \in \bigcup_{l=1}^n D^{E_j^-}_l$, which happens only when $\alpha_0 = \alpha_1 = 0$.  Moreover $\phi^{E_j^-}_\Delta(1) = (1, \ldots, 1)$, and so $\theta_0 = \cdots = \theta_n = 1$.  Thus $\phi^{E_j^-}_\Delta =  (z, \phi_\Delta(z))$ in the local coordinates $(w, \chi)$, where $\phi_\Delta(z) = (z, 1, \ldots, 1)$.  Thus $\phi^{E_j^-}$ comes from the stable disc $\phi$ in $X$, which is a union of $\phi_\Delta$ and $\phi_C$.

\noindent{\bf (C) Kuranishi Structures.} Now we compare the Kuranishi structures on the both sides of \eqref{open=open_K_spaces}.  Let us have a brief reasoning on why they should have the same Kuranishi structures.  On both sides the disc components are regular, and so the obstructions merely come from the rational curve components in class $\alpha$.  For the curve component of $\phi^{E_j^-}$, since it is free to move from fiber to fiber of ${E_j^-} \to \proj^1$, the obstruction comes from the directions along $X$, and this is identical with the corresponding curve component of $\phi$.  Now consider the deformations.  Due to the boundary point condition, the disc components on both sides cannot be deformed.  For the curve component of $\phi^{E_j^-}$, the interior point condition that it has to pass through $D^{E_j^-}$ kills the deformations in the direction transverse to fibers. Thus $\phi^{E_j^-}$ has the same deformations as $\phi$.  Therefore the corresponding stable discs on both sides have the same deformations and obstructions, and hence the moduli have the same Kuranishi structures.  In what follows, we write down and equate the deformations and obstructions explicitly on both sides.

A Kuranishi structure on $\moduli^{\mathrm{op}}_{1,1}(\beta_j+\alpha; D;\pt_{\tor^X})$ assigns a Kuranishi chart
$$(V_{\mathrm{op}}, \mathscr{E}^-_{\mathrm{op}}, \Gamma_{\mathrm{op}}, \psi_{\mathrm{op}}, s_{\mathrm{op}})$$
around each  $\phi \in \moduli^{\mathrm{op}}_{1,1}(\beta_j+\alpha; D;\pt_{\tor^X})$ which is constructed as follows.  Let $$D_{\phi} \bar{\partial}: W^{1,p}(\mathrm{Dom}(\phi), \phi^*(TX), \tor) \to W^{0,p}(\mathrm{Dom}(\phi), u^*(TX) \otimes \Lambda^{0,1})$$
be the linearized Cauchy-Riemann operator at $\phi$. (Here $\mathrm{Dom}(\phi)$ is the domain of $\phi$.)

\begin{enumerate}
\item
$\Gamma_{\mathrm{op}}$ is the automorphism group of $\phi$, that is, the group of all elements
$$g \in \Aut(\Dom(\phi),p^{\mathrm{int}},p^{\mathrm{bdy}})$$
such that $\phi \circ g = \phi$.  By stability of $\phi$, $\Gamma_{\mathrm{op}}$ is a finite group. (Note that by definition, $g(p^{\mathrm{int}})= p^{\mathrm{int}}, g(p^{\mathrm{bdy}})= p^{\mathrm{bdy}}$.)

\item
The so-called obstruction space $\mathscr{E}^-_{\mathrm{op}}$ is the cokernel of the linearized Cauchy-Riemann operator $D_{\phi} \bar{\partial}$, which is finite dimensional since $D_{\phi} \bar{\partial}$ is Fredholm.  For the purpose of the next step of construction, it is identified (in a non-canonical way) with a subspace of $W^{0,p}(\mathrm{Dom}(\phi), \phi^*(TX) \otimes \Lambda^{0,1})$ as follows.  Denote by $\Delta$ and $S_1, \ldots, S_l$ the disc and sphere components of $\mathrm{Dom}(\phi)$ respectively.   Take non-empty open subsets $W_0 \subset \Delta$ and $W_i \subset S_i$ for $i = 1, \ldots, l$.  Then by unique continuation theorem there exists finite dimensional subspaces $\mathscr{E}^-_i \subset C_0^\infty (W_i, \phi^*(TX)\otimes \Lambda^{0,1})$ such that
$$\mathrm{Im} (D_{\phi} \bar{\partial}) \oplus \mathscr{E}^-_{\mathrm{op}} = W^{0,p}(\mathrm{Dom}(\phi), \phi^*(TX) \otimes \Lambda^{0,1})$$
and $\mathscr{E}^-_{\mathrm{op}}$ is invariant under $\Gamma_{\mathrm{op}}$, where $$\mathscr{E}^-_{\mathrm{op}} := \mathscr{E}^-_0 \oplus \cdots \oplus \mathscr{E}^-_l.$$

\item
$\tilde{V}_{\mathrm{op}}$ is taken to be (a neighborhood of $0$ of) the space of first order deformations $\Phi$ of $\phi$ which satisfies the linearized Cauchy-Riemann equation modulo elements in $\mathscr{E}^-_{\mathrm{op}}$:
$$ (D_{\phi} \bar{\partial}) \cdot \Phi \equiv 0 \,\, \mod \mathscr{E}^-_{\mathrm{op}}. $$
Such deformations may come from deformations of the map or deformations of complex structures of the domain.  More precisely,
$$\tilde{V}_{\mathrm{op}} = V_{\mathrm{op}}^{\mathrm{map}} \times V_{\mathrm{op}}^{\mathrm{dom}}$$
where $V_{\mathrm{op}}^{\mathrm{map}}$ is defined in the following way.  Let $V'_{\mathrm{op},\mathrm{map}}$ be the kernel of the linear map $$[D_{\phi} \bar{\partial}]: W^{1,p}(\mathrm{Dom}(\phi), \phi^*(TX), \tor) \to W^{0,p}(\mathrm{Dom}(\phi), \phi^*(TX) \otimes \Lambda^{0,1})/\mathscr{E}^-_{\mathrm{op}}.$$
Notice that $\Aut(\mathrm{Dom}(\phi),p^{\mathrm{int}},p^{\mathrm{bdy}})$ may not be finite since the domain of $\phi$ may not be stable, and it acts on $V'_{\mathrm{op},\mathrm{map}}$.  Thus its Lie algebra $\mathfrak{g}$ is contained in $V'_{\mathrm{op},\mathrm{map}}$, and we take $V_{\mathrm{op}}^{\mathrm{map}} \subset V'_{\mathrm{op},\mathrm{map}}$ such that
$V'_{\mathrm{op},\mathrm{map}} = V_{\mathrm{op}}^{\mathrm{map}} \oplus \mathfrak{g}$.

$V_{\mathrm{op}}^{\mathrm{dom}}$ is a neighborhood of zero in the space of deformations of the domain rational curve $C$.  Such deformations consists of two types: one is deformations of each stable component (in this genus 0 case, it means movements of special points in each component), and another one is smoothing of nodes between components.  That is,
$$V_{\mathrm{op}}^{\mathrm{dom}} = V_{\mathrm{op}}^{\mathrm{cpnt}} \times V_{\mathrm{op}}^{\mathrm{smth}}$$
where $V_{\mathrm{op}}^{\mathrm{cpnt}}$ is a neighborhood of zero in the space of deformations of components of $C$, and $V_{\mathrm{op}}^{\mathrm{smth}}$ is a neighborhood of zero in the space of smoothing of the  nodes (each node contributes to a one-dimensional family of smoothings).  Each deformation in $V_{\mathrm{op}}^{\mathrm{dom}}$ gives $\Delta \cup \tilde{C}$, where $\Delta$ is a disc with one boundary marked point, and $\tilde{C}$ is a rational curve with one interior marked point, such that $\Delta$ and $\tilde{C}$ intersect at a nodal point.  $\Delta \cup \tilde{C}$ serves as the domain of the deformed map $\Phi$.

\item
$\tilde{s}_{\mathrm{op}}:\tilde{V}_{\mathrm{op}} \to \mathscr{E}^-_{\mathrm{op}}$ is a transversal $\Gamma_{\mathrm{op}}$-equivariant perturbed zero-section of the trivial bundle $\mathscr{E}^-_{\mathrm{op}} \times \tilde{V}_{\mathrm{op}}$ over $\tilde{V}_{\mathrm{op}}$.  By \cite{FOOO1}, this can be chosen to be $\tor$-equivariant.

\item
There exists a continuous family of smooth maps $\rho^{\mathrm{op}}_{\Phi}: (\mathscr{D},\partial\mathscr{D}) \to (X,\tor)$ over $\tilde{V}_{\mathrm{op}} \ni \Phi$ such that it solves the inhomogeneous Cauchy-Riemann equation: $\bar{\partial} \rho^{\mathrm{op}}_{\Phi} = \tilde{s}_{\mathrm{op}}(\Phi).$  Set
$$V_{\mathrm{op}} := \{\Phi \in \tilde{V}_{\mathrm{op}}: \mathrm{ev}_0 (\rho^{\mathrm{op}}_{\Phi}) = \pt_{\tor^X}; \mathrm{ev}_+ (\rho^{\mathrm{op}}_{\Phi}) \in D \}$$
where $\mathrm{ev}_0$ is the evaluation map at $p^{\mathrm{bdy}}$.  Then set $s_{\mathrm{op}} := \tilde{s}_{\mathrm{op}}|_{V_{\mathrm{op}}}$.

\item
$\psi_{\mathrm{op}}$ is a map from $s_{\mathrm{op}}^{-1} (0) / \Gamma_{\mathrm{op}}$ onto a neighborhood of $[\phi] \in \moduli^{\mathrm{op}}_{1,1}(\beta_j+\alpha; D; \mathrm{pt}_{\tor^X})$.

\end{enumerate}

Now comes the key: in Item (2) of the above construction, since the disc component of $\phi$ is unobstructed (that is, the linearized Cauchy-Riemann operator localized to the disc component is surjective), $\mathscr{E}^-_0 = 0$ so that $\mathscr{E}^-_{\mathrm{op}}$ is of the form $$\mathscr{E}^-_{\mathrm{op}} = 0 \oplus \mathscr{E}^-_1 \oplus \cdots \oplus \mathscr{E}^-_l.$$  The analogous statement is also true for the corresponding stable disc $\phi^{E_j^-} \in \moduli^{\mathrm{op}}_{1,1}(b_0+b_j+\alpha; D^{E_j^-};\pt_{\tor^{E_j^-}})$.  With this observation, we argue in the following that $(V_{\mathrm{op}}, \mathscr{E}^-_{\mathrm{op}}, \Gamma_{\mathrm{op}}, \psi_{\mathrm{op}}, s_{\mathrm{op}})$ can be identified as a Kuranishi chart $(V^E_{\mathrm{op}}, \mathscr{E}_{\mathrm{op}}^{-,E}, \Gamma^E_{\mathrm{op}}, \psi^E_{\mathrm{op}}, s^E_{\mathrm{op}})$ around the corresponding stable disc $\phi^{E_j^-}$ bounded by $\tor^{E_j^-} \subset {E_j^-}$.

\begin{enumerate}
\item
$\phi$ and $\phi^{E_j^-}$ have the same automorphism group, that is, $\Gamma_{\mathrm{op}} = \Gamma^E_{\mathrm{op}}$.  This is because the disc component have only one boundary marked point and one interior nodal point and thus has no automorphism, and any automorphism on the rational-curve part of $\phi$ will give an automorphism on the rational-curve part of $\phi^{E_j^-}$, and vice versa.

\item
The disc component of $\phi^{E_j^-}$ is unobstructed.  For the rational curve component $C$ which is mapped into $\mathscr{D}_0 \cong X$, notice that there is a splitting $TE|_{\mathscr{D}_0} = T(\mathscr{D}_0) \oplus N\mathscr{D}_0$ and so $W^{0,p}(C, (\phi^{E_j^-}|_C)^*(TE) \otimes \Lambda^{0,1})$ is equal to
$$
\hspace{35pt} W^{0,p}(C, (\phi^{E_j^-}|_C)^*(T\mathscr{D}_0) \otimes \Lambda^{0,1}) \oplus W^{0,p}(C, (\phi^{E_j^-}|_C)^*(N\mathscr{D}_0) \otimes \Lambda^{0,1})
$$
where the first summand is equal to $W^{0,p}(C, (\phi)^*TX \otimes \Lambda^{0,1})$.

Since the curve component is free to move in the direction of the normal bundle $N\mathscr{D}_0$, we have
$$\Image (D_{\phi^{E_j^-}|_C} \bar{\partial}) \supset W^{0,p}(C, (\phi^{E_j^-}|_C)^*(N\mathscr{D}_0) \otimes \Lambda^{0,1}).$$
Hence
$$\hspace{35pt} \Image (D_{\phi^{E_j^-}} \bar{\partial}) \oplus (0 \oplus \mathscr{E}^-_1 \oplus \cdots \oplus \mathscr{E}^-_l) = W^{0,p}(\mathrm{Dom}(\phi^{E_j^-}), (\phi^{E_j^-})^*(TX) \otimes \Lambda^{0,1}). $$
Thus we may take $\mathscr{E}^{-,E}_{\mathrm{op}} = 0 \oplus \mathscr{E}^-_1 \oplus \cdots \oplus \mathscr{E}^-_l$.

\item
$\tilde{V}^E_{\mathrm{op}}$, $\tilde{s}^E_{\mathrm{op}}$, $\rho^{{E_j^-},\mathrm{op}}_{\Phi^{E_j^-}}$ are defined in the same way as above.
The subspace $\tilde{V}'$ of those deformations $\Phi^{E_j^-} \in \tilde{V}^E_{\mathrm{op}}$ such that the image of the curve component under $\rho^{{E_j^-},\mathrm{op}}_{\Phi^{E_j^-}}$ lies in $\mathscr{D}_0$ is isomorphic to $\tilde{V}_{\mathrm{op}}$, and restrictions of $\tilde{s}^E_{\mathrm{op}}$ and $\rho^{{E_j^-},\mathrm{op}}_{\Phi^{E_j^-}}$ to $\tilde{V}'$ gives choices of $\tilde{s}_{\mathrm{op}}$ and $\rho^{\mathrm{op}}$ respectively.  Moreover,
$$ V^E_{\mathrm{op}} := \{\Phi^{E_j^-} \in \tilde{V}^E_{\mathrm{op}}: \mathrm{ev}_0 (\rho^{{E_j^-},\mathrm{op}}_{\Phi^{E_j^-}}) = \pt_{\tor^{E_j^-}}; \mathrm{ev}_+ (\rho^{{E_j^-},\mathrm{op}}_{\Phi^{E_j^-}}) \in D^{E_j^-} \} $$
lies in $\tilde{V}_{\mathrm{op}}$.  Thus $V^E_{\mathrm{op}} = V_{\mathrm{op}}$, and $s^E_{\mathrm{op}} := \tilde{s}^E_{\mathrm{op}}|_{V^E_{\mathrm{op}}} = s_{\mathrm{op}}$.
Then
$\psi_{\mathrm{op}}$ can be identified as a map $\psi^E_{\mathrm{op}}$ which maps $(s^E_{\mathrm{op}})^{-1} (0) / \Gamma_{\mathrm{op}}^{E}$ onto a neighborhood of $[\phi^{E_j^-}] \in \moduli^{\mathrm{op}}_{1,1}(b_0+b_j+\alpha; D^{E_j^-};\pt_{\tor^{E_j^-}})$.
\end{enumerate}
In conclusion, a Kuranishi neighborhood of $\phi$ can be identified with a Kuranishi neighborhood of $\phi^{E_j^-}$.  Thus the Kuranishi structures on $\moduli^{\mathrm{op}}_{1,1}(\beta; D, \pt_{\tor^X})$ and that on $\moduli^{\mathrm{op}}_{1,1}(b_0+b_j+\alpha; D^{E_j^-}, \pt_{\tor^{E_j^-}})$ are identical. This completes the proof of the proposition.
\end{proof}

\subsection{Second step}
Now we come to the second main step, which is the following theorem:

\begin{thm} \label{Thm {E_j^-}-open-closed}
Assume the notations as in Theorem \ref{Thm open-closed}, $v_i \in F(v_j)$ and $D_l \cdot \alpha = 0$ whenever $v_l \not\in F(v_j)$.  Then
$$ n^{{E_j^-}}_{1,1}(b_0 + b_j + \alpha; D^{{E_j^-}}, [\pt]_{\tor^{{E_j^-}}}) = \langle D^{E_j^-}, [\pt]_{{E_j^-}} \rangle^{E_j^-,\,\sigma_j^-\,\reg}_{0,2,\sigma_j^-+\alpha}. $$
\end{thm}

By Equation \eqref{eqn open_E},
$$n^{{E_j^-}}_{1,1}(b_0 + b_j + \alpha; D^{{E_j^-}}, [\pt]_{\tor^{{E_j^-}}}) = [\moduli^{\mathrm{op}}_{1,1}(b_0+b_j+\alpha; D^{E_j^-}, \pt_{\tor^{E_j^-}})]_{\mathrm{virt}} \in H^0(D^{E_j^-} \times \{\pt_{\tor^{E_j^-}}\}, \rat) = \rat$$
and
$$\langle D^{{E_j^-}}, [\pt]_{{E_j^-}} \rangle^{E_j^-,\sigma_j^- \,\reg}_{0,2,\sigma_j^-+\alpha} = [\moduli^{\mathrm{cl},\sigma_j^- \,\reg}_{0,2,\sigma_j^-+\alpha}(D^{E_j^-},\pt)]_{\mathrm{virt}} \in H^0(\pt, \rat) = \rat$$
is the $\sigma_j^-$-regular closed GW invariants given in Definition \ref{def:regGW}.  In order to prove the equality between open and closed invariants in Theorem \ref{Thm {E_j^-}-open-closed}, it suffices to exhibit an isomorphism between the Kuranishi structures:

\begin{prop} \label{Prop open-closed mod}
Assume the condition in Theorem \ref{Thm {E_j^-}-open-closed}.  Fix a point $\pt \in \tor^{E_j^-} \subset {E_j^-}$.  Then
\begin{equation}\label{open=closed_K_spaces}
\moduli^{\mathrm{op}}_{1,1}(b_0+b_j+\alpha; D^{E_j^-}, \pt_{\tor^{E_j^-}}) \cong \moduli^{\mathrm{cl},\,\sigma_j^-\,\reg}_{0,2,\sigma_j^- + \alpha}(D^{E_j^-},\pt).
\end{equation}
as Kuranishi spaces.
\end{prop}

The proof is very similar to that of Proposition \ref{prop open mod compare}: we first prove that the two sides are equal as sets, and then compare the Kuranishi charts and show that they can be chosen to be the same.

First, let us consider the case $\alpha = 0$ and $i \neq j$.
We have seen in the proof of Proposition \ref{prop open mod compare} that the LHS of \eqref{open=closed_K_spaces} is the empty set when $\alpha = 0$ and $i \neq j$.  For the right-hand side, we have the following lemma:
\begin{lem} \label{lem:delta}
For $i\not=j$,
the moduli space $\moduli^{\mathrm{cl},\sigma_j^- \reg}_{0,2,\sigma_j^-}(D_i^{E_j^-},\pt)$ is empty.  In particular we have
$$ \langle D_i^{E_j^-}, \pt \rangle^{E_j^-, \sigma_j^- \reg}_{0,2,\sigma_j^-} = 0.$$
\end{lem}
\begin{proof}
By Lemma \ref{lem:reg=>unique}, a rational curve in $\moduli^{\mathrm{cl},\sigma_j^- \reg}_{0,2,\sigma_j^-}(D_i^{E_j^-},\pt)$ is a holomorphic sphere representing $\sigma_j^-$ passing through $\pt$ in the open toric orbit.  Such a sphere is unique and intersect $\cD_0$ at only one point which lies in $D_j^{E_j}$.  It never intersects $D_i^{E_j^-}$ for $i\not=j$.  Hence the moduli space is empty.
\end{proof}

By the above lemma, when $\alpha = 0$ and $i \neq j$, both sides are the empty set, and we have
$n^{{E_j^-}}_{1,1}(b_0 + b_j; D_i^{{E_j^-}}, [\pt]_{\tor^{{E_j^-}}}) = \langle D_i^{{E_j^-}}, [\pt]_{{E_j^-}} \rangle^{E_j^-,\,\sigma_j^-\,\reg}_{0,2,\sigma_j^-} = 0$.

\noindent\textbf{Proof of Proposition \ref{Prop open-closed mod} when $\alpha \neq 0$ or $i = j$.}  First we construct a bijection between the two sides of \eqref{open=closed_K_spaces}.

For a stable disc in $\moduli^{\mathrm{op}}_{1,1} (b_0 + b_j + \alpha) \times_{{E_j^-} \times \tor^{E_j^-}} (D_i^{E_j^-} \times \{\pt\})$, we denote the domain interior marked point by $p^{\mathrm{int}}$ and the domain boundary marked point by $p^{\mathrm{bdy}}$.  For a rational curve in $\moduli^{\mathrm{cl},\,\sigma_j^-\,\reg}_{0,2,\sigma_j^- + \alpha}  \times_{{E_j^-} \times {E_j^-}} (D_i^{E_j^-} \times \{\pt\})$, we denote by $p_0$ the marked point mapped to $D_i^{E_j^-}$, and $p_1$ the marked point mapped to $\pt$.

By relabeling $\{D_l\}_{l=1}^m$ if necessary, we assume $j = 1$.  Let us fix a local toric chart $\chi = (\chi_1, \ldots, \chi_n)$ of $X$ which covers the open orbit of $D_1 \subset X$, and such that $\chi_l|_{D_l} = 0$ (by relabeling $\{D_l\}_{l=1}^m$ if necessary). Correspondingly we have the local chart $(\chi,w)$ of ${E_j^-}$ around the fiber $w=0 \in \proj^1$.  Without loss of generality we take $\tor^{E_j^-}$ to be the fiber $|\chi_l| = |w| = 1$ for all $l$, and $\pt \in \tor^{E_j^-}$ to be $\chi_l(\pt) = w(\pt) = 1$ for all $l$.

Consider the case when $\alpha = 0$ and $i=j$.  We have seen in the proof of Proposition \ref{prop open mod compare} that the LHS of \eqref{open=closed_K_spaces} is a singleton when $i = j$.  When $i=j$, it is the disc $w = z, \chi = (z,1, \ldots, 1)$ on $\Delta \ni z$.

On the RHS of \eqref{open=closed_K_spaces}, by Lemma \ref{lem:reg=>unique} the element is the unique holomorphic sphere $\rho$ representing $\sigma_j^-$ passing through $\pt$ and $D_j^{E_j^-}$.  Since $\mathscr{D}_\infty \cdot \sigma_j^-  = 1$, there is a unique point $p_\infty \in \proj^1$ with $\rho(p_\infty) \in \mathscr{D}_\infty$.  By composing with an automorphism of $\proj^1$, we may assume $p_0 = 0$, $p_1 = 1$ and $p_\infty = \infty$.  Consider $\chi_l \circ \rho, w \circ \rho: \proj^1 \to \proj^1$ for $l = 1, \ldots, n$.  Since $\mathscr{D}_l \cdot \sigma_j^- = 0$ except when $l = 0, j, \infty$, $\chi_l \circ \rho$ are constants for $l = 2, \ldots, n$.  Thus $\chi_l \circ \rho = \chi_l \circ \rho (p_1) = 1$ for $l = 2, \ldots, n$.  Moreover $\chi_1 \circ \rho (z)$ and $w \circ \rho (z)$ have only one zero at $0$ and one pole at $\infty$, and so they are equal to $c z$ for some $c \in \cpx$.  But $\rho (p_1) = \pt$ implies $\chi_1 \circ \rho (p_1) = 1$, and this forces $c=1$.  Thus $\rho$ is $(w,\chi) = (z,(z,1,\ldots,1))$.  The curve is regular and so the obstruction is trivial.  This proves that for the case when $i=j$ and $\alpha=0$, we have the following
\begin{lem} \label{lem:singleton}
The moduli space $\moduli^{\mathrm{cl},\,\sigma_j^-\,\reg}_{0,2,\sigma_j^-}(D_j^{E_j^-}, \pt)$ is a singleton, and we have
$$ \langle \pt, D_j^{E_j^-} \rangle^{E_j^-,\,\sigma_j^-\,\reg}_{0,2,\sigma_j^-} = 1.$$
\end{lem}

In particular there is a bijection between the LHS and RHS of \eqref{open=closed_K_spaces}.

Now consider the case when $\alpha \neq 0$.  Let $\phi_{\mathrm{op}}^{E_j^-}$ be a stable disc in the LHS.  From the proof of Proposition \ref{prop open mod compare}, $\phi_{\mathrm{op}}^{E_j^-}$ is a holomorphic disc representing $b_0 + b_j$ attached with a rational curve $C$ representing $\alpha$ at exactly one nodal point, where the interior marked point $p^{\mathrm{int}}$ is located in $C$, and the map $\phi_{\mathrm{op}}^{E_j^-}|_\Delta$ on the disc component $\Delta \ni z$ is given by $(w, \chi) = (z, (z, 1, \ldots, 1))$.  Such a map from $\Delta$ to ${E_j^-}$ analytically extends to a map $\phi_{\mathrm{cl},\proj^1}: \proj^1 \to {E_j^-}$, where $\infty \in \proj^1$ is mapped to $w = \infty$ and $\chi_l = 1$ for $l = 2, \ldots, n$, which is the point $\pt \in \mathscr{D}_\infty$.  Then $\phi_{\mathrm{cl},\proj^1}$ attached with the same rational curve $C$, with marked points $p_0 = p^{\mathrm{int}}$ in the rational curve and $p_1 = \infty \in \Dom (\phi_{\mathrm{cl},\proj^1})$, is an element in the moduli on the RHS.  This gives a map from the LHS to the RHS of \eqref{open=closed_K_spaces}.

Now we show that this map is invertible.  By Lemma \ref{lem:reg=>unique}, an element in $\moduli^{\mathrm{cl},\,\sigma_j^-\,\reg}_{0,2,\sigma_j^- + \alpha}(D_i^{E_j^-},\pt)$ is the unique holomorphic sphere $\rho_{\proj^1}: \proj^1 \to {E_j^-}$ representing $\sigma_j^-$ passing through $\pt$ and $D_j^{E_j^-}$ union with a rational curve $\rho_C: C \to {E_j^-}$ representing $\alpha$.  By the above argument, $(w,\chi)\circ \rho_{\proj^1}(z) = (z,(z,1,\ldots,1))$ where $\rho_{\proj^1}(0)$ is the nodal point.  Then by restricting  $\rho_{\proj^1}$ to $\Delta \subset \proj^1$, we obtain a stable disc in the LHS.  This gives the inverse of the above map.

The comparison of Kuranishi structures is very similar to the proof of Proposition \ref{prop open mod compare} and thus omitted (cf. \cite{Chan10,LLW10}).

\section{Computing closed invariants by Seidel representations} \label{sect compute Seidel}

In this section we prove Theorems \ref{thm delta_intro}, \ref{thm W equal}, \ref{thm conv}, \ref{thm Seidel to gen intro} as promised in the introduction.

\subsection{Calculations}
We have equated the open GW invariants appearing in the disc potential of $X$ with certain two-point closed GW invariants in the Seidel spaces $E_j^-$ associated to $X$.  Computing these closed GW invariants is challenging. Firstly, these closed GW invariants are more refined closed GW invariants, namely $\sigma^{-}$-regular GW invariants as defined in Definition \ref{def:regGW}. Also, because the infinity section class $\sigma^-_\infty \in H_2^{\eff}(E_j^-)$ may have $c_1(\sigma^-_\infty) < 0$, $E_j^-$ is {\em not} semi-Fano.  This is because the $\cpx^*$-action induced by $-v_j$ can have a fixed locus whose normal bundle has total weights less than $-2$. Thus, many tools such as the mirror theorem do not apply to our setting. 

Our computation of open GW invariants involves a number of techniques. Observe that the Seidel space $E_j$ associated to $v_j$ is always semi-Fano because every fixed locus in $X$ has total weights not less than $-2$ (the fixed locus $D_j$ has weight $-1$ which is already minimum), see \cite[Lemma 3.2]{G-I11}.  In this case the mirror theorem for $E_j$ is much easier to handle.  In particular, the normalized Seidel element $S^\circ_j$ corresponding to $E_j$ has been computed by Gonz\'alez-Iritani \cite{G-I11} and can be explicitly expressed in terms of the Batyrev element $B_j$:

\begin{prop}[\cite{G-I11}, Theorem 3.13 and Lemma 3.17 and \cite{G-I12}, Remark 4.18] \label{prop Seidel Batyrev}
$$ B_j(\check{q}(q)) = \exp (g_j(\check{q}(q))) S^\circ_j(q) = D_j - \sum_{i=1}^{m} g_{i,j}(\check{q}(q)) D_i,$$
where
$$g_{i,j}(\check{q}):=\sum_{d}\frac{(-1)^{(D_j\cdot d)}(D_j,d)(-(D_i\cdot d)-1)!}{\prod_{p\neq i} (D_p\cdot d)!}\check{q}^d,$$
where the summation is over all effective curve classes $d \in H_2^\text{eff}(X)$ satisfying
$-K_X\cdot d=0$, $D_i\cdot d<0$ and  $D_p\cdot d \geq 0$ for all $p\neq i$.
\end{prop}

By Theorem \ref{Thm open-closed}, the open invariants are equal to the closed invariants $\langle \iota_* D_i, [\pt] \rangle^{E_j^-,\,\sigma_j^-\,\reg}_{0,2,\sigma_j^-+\alpha}$, where $v_i \in F(v_j)$ and $\alpha \in H_2^{\eff,c_1=0}(X)$ is such that $D_l \cdot \alpha = 0$ for $v_l\notin F(v_j)$.  To compute them it is useful to express $D_i$ in terms of the Seidel elements $S_l$'s.


\begin{prop} \label{prop:li}
For every $i$, $\{B_i\} \cup \{B_l: g_l \neq 0 \}$ is a linearly independent set in $H^*(X,\rat)$.
\end{prop}
\begin{proof}
It can be seen from the fan polytope of $X$.  Since $X$ is semi-Fano, the generators $v_l$ of rays lie on the boundary of the fan polytope, and those $v_l$ with $g_l \neq 0$ are not the vertices of the fan polytope by \cite[Proposition 4.3]{G-I11}.  Since the number of vertices is at least $n+1$, the number of $v_l$'s with $g_l \neq 0$ is no more than $m-n-1$.  Moreover the only relations among the $B_l$'s (regarded as elements in a vector space) are the linear relations, and all of them involve elements outside $\{B_l: g_l \neq 0 \}$.  Thus $\{B_l: g_l \neq 0 \}$ is linearly independent.  Every linear relation involves more than two vertices, and hence there is no linear relation involving only $B_i$ and $B_l$'s with $g_l \neq 0$.
\end{proof}

\begin{prop} \label{Prop D-B}
We have
$$D_i = \tilde{B}_i + \sum_{l=1}^m (\hat{D}_i \cdot g_l(\check{q}(q))) \tilde{B}_l$$
as divisors (where $\tilde{B}_l$'s are the extended Batyrev elements in Definition \ref{def:extB}).  Thus
$[D_i] = B_i + \sum_{l=1}^m (\hat{D}_i \cdot g_l(\check{q}(q))) B_l$
as elements in $\QH^*(X)$.
\end{prop}
\begin{proof}
This follows directly from the definition of the extended mirror map $\log Q_l(\check{Q}) = \log \check{Q}_l - g_l(\check{q}(\check{Q}))$ and definition of the extended Batyrev elements as push forward of the basis $\{D_1,\ldots,D_m\} \subset H^2(X,T)$ via the differential of the extended mirror map.
\end{proof}

By Propositions \ref{prop Seidel Batyrev}, the (normalized) Seidel elements can be taken to be the divisors\footnote{These were defined to be the lifts of Seidel elements in \cite{G-I12}.} $S_l^\circ = \exp (-g_l(\check{q}(q))) \tilde{B}_l$. Then by Proposition \ref{Prop D-B}, we have
$$D_i = \exp (g_i(\check{q}(q))) S_i^\circ + \sum_{l=1}^m (\hat{D}_i \cdot g_l(\check{q}(q))) \exp (g_l(\check{q}(q))) S_l^\circ $$
as divisors.  Then
\begin{equation} \label{eq:main}
\begin{aligned}
&\langle \iota_* D_i, [\pt] \rangle^{E_j^-,\,\sigma_j^-\,\reg}_{0,2,\sigma_j^-+\alpha}\\
=&\exp (g_i(\check{q}(q))) \langle \iota_* S^\circ_i, [\pt] \rangle^{E_j^-,\,\sigma_j^-\,\reg}_{0,2,\sigma_j^-+\alpha} + \sum_{l=1}^m (\hat{D}_i \cdot g_l(\check{q}(q))) \exp (g_l(\check{q}(q))) \langle \iota_* S^\circ_l, [\pt] \rangle^{E_j^-,\,\sigma_j^-\,\reg}_{0,2,\sigma_j^-+\alpha}.
\end{aligned}
\end{equation}


\begin{prop} \label{prop:delta}
We have
\begin{equation}\label{eqn:2pt_reg_inv}
\sum_\alpha q^\alpha \langle \iota_* S^\circ_i, [\pt] \rangle^{E_j^-,\,\sigma_j^-\,\reg}_{0,2,\sigma_j^-+\alpha} = \delta_{ij}.
\end{equation}
\end{prop}
\begin{proof}
The idea is to use the degeneration family, which is the key to derive the composition law $S_{v_i-v_j} = S_{v_i} * S_{-v_j}$ of Seidel representation, and restrict it to those connected components of the moduli which contribute to $\langle \iota_* S^\circ_i, [\pt] \rangle^{E_j^-,\,\sigma_j^-\,\reg}_{0,2,\sigma_j^-+\alpha}$.  We use the degeneration due to McDuff \cite{McDuff_seidel}; degenerations for Seidel representations were also extensively studied in \cite[Section 29]{FOOO-spec}.

Consider the degeneration family of $E_{v_i-v_j}$ to a union of $E_{v_i}$ and $E_{-v_j}$ along $X$ (see Equation \eqref{eq:degfam}).  It gives a degeneration formula as follows.  By the construction of McDuff \cite[Sections 2.3.2 and 4.3.3]{McDuff_seidel}, there is a family $\cF$ of moduli spaces over the disc whose generic fiber is $\cM_{0,4,\sigma_i + \sigma_j^- + \alpha}^{E_{v_i-v_j}}$ and whose fiber at zero is $\cF_0 = \bigcup_{s_1 + s_2=\sigma_i + \sigma_j^- + \alpha} \cM^{E_{v_i}}_{0,3,s_1} \times_X \cM^{E_{-v_j}}_{0,3,s_2}$.  Let $\pt$ be a generic point chosen such that in the degeneration, $\pt$ lies in the open toric orbit of $E_{-v_j}$.  Let $X_0, X_z, X_{z'}$ be fibers of $E_{v_i-v_j} \to \bP^1$ for $0,z,z' \in \bP^1$ (which are isomorphic to $X$) such that in the degeneration, $X_0, X_z \subset E_{v_i}$ are the fibers of $E_{v_i} \to \bP^1$ at $0, \infty \in \bP^1$ and $X_{z'} \subset E_{-v_j}$ is a fiber of $E_{-v_j} \to \bP^1$ at a generic point.  Taking the fiber product with $X_0$, $X_z$, $X_{z'}$ and the generic point $\pt$, we get a family $\cF(\pt)$ whose generic fiber is
$\cM_{0,4,\sigma_i + \sigma_j^- + \alpha}^{E_{v_i-v_j}}(X_0,X_z,X_{z'},\pt)$
and whose fiber at zero is
$$\cF_0(\pt) = \bigcup_{s_1 + s_2=\sigma_i + \sigma_j^- + \alpha} \cM^{E_{v_i}}_{0,3,s_1}(X_0,X_z) \times_X \cM^{E_{-v_j}}_{0,3,s_2}(X_{z'},\pt).$$

Let $\{\Phi_l\}$ be a basis of $H^*(X)$ and $\{\Phi^l\}$ be the dual basis with respect to the Poincar\'e pairing.
Then the degeneration formula in \cite[Sections 2.3.2 and 4.3.3]{McDuff_seidel} gives
$$
\begin{aligned}
&\langle [\pt] \rangle_{0,1,\sigma_i + \sigma_j^- + \alpha}^{E_{v_i-v_j}} = \langle X_0, X_z, X_{z'}, [\pt] \rangle_{0,4,\sigma_i + \sigma_j^- + \alpha}^{E_{v_i-v_j}}\\
=& \sum_{\substack{s_1 + s_2=\sigma_i + \sigma_j^- + \alpha\\l}}\langle X_0,X_z,\iota_*\Phi_l \rangle_{0,3,s_1}^{E_{v_i}} \langle X_{z'},\iota_*\Phi^l,[\pt] \rangle^{E_{-v_j}}_{0,3,s_2} = \sum_{\substack{s_1 + s_2=\sigma_i + \sigma_j^- + \alpha\\l}}\langle \iota_*\Phi_l \rangle_{0,1,s_1}^{E_{v_i}} \langle \iota_*\Phi^l,[\pt] \rangle^{E_{-v_j}}_{0,2,s_2},
\end{aligned}
$$
where the first and last equality follows from the divisor equation (a section class intersects a fiber class once).  The left-hand side is one of the terms of $\pairing{S_{v_i-v_j}}{[\pt]}$, while the right-hand side are terms appearing in $\pairing{S_{v_i} * S_{-v_j}}{[\pt]}$.  The degeneration formula is the main ingredient in deriving the composition law $S_{v_i-v_j} = S_{v_i} * S_{-v_j}$.

Each fiber of $\cF(\pt)$ is compact and has finitely many connected components.  We denote by $\cF_0(\pt)^{\sigma_j^- \,\reg}$ the union of those connected components of $\cF_0(\pt)$ which contain a rational curve with a sphere component in $E_{-v_j}$ representing $\sigma_j^-$:
\begin{align*}
\cF_0(\pt)^{\sigma_j^- \,\reg} &= \bigcup_{s_1 + s_2=\sigma_i + \sigma_j^- + \alpha} \cM^{E_{v_i}}_{0,3,s_1}(X_0,X_z) \times_X \cM^{E_{-v_j},\sigma_j^- \,\reg}_{0,3,s_2}(X_{z'},\pt) \\
&= \bigcup_{s_1 + s_2=\sigma_i + \sigma_j^- + \alpha} \left(\cM^{E_{v_i}}_{0,3,s_1}(X_0,X_z) \times_{E_{v_i}} \cD^{E_{v_i}}_\infty\right) \times_{E_{-v_j}} \cM^{E_{-v_j},\sigma_j^- \,\reg}_{0,3,s_2}(X_{z'},\pt).
\end{align*}
The virtual cycle of the above expression is (locally) the zeroes of $(s_1,s_2)$ (modding out finite automorphisms), where $s_1,s_2$ are multi-sections of the first and second factors respectively.  The zeroes of $s_1$ give the virtual cycle $[\cM^{E_{v_i}}_{0,3,s_1}(X_0,X_z,\cD^{E_{v_i}}_\infty)]_\virt$ of the first factor, which is a cycle in $\cD^{E_{v_i}}_\infty \cong \cD_0^{E_{-v_j}}$, fiber product with the second factor $\cM^{E_{-v_j},\sigma_j^- \,\reg}_{0,3,s_2}(X_{z'},\pt)$.  Then the zeroes of $s_2$ gives the virtual cycle
\begin{align*}
[\cF_0(\pt)^{\sigma_j^- \, \reg}]_\virt 
= \sum_{s_1 + s_2=\sigma_i + \sigma_j^- + \alpha} \left[\cM^{E_{-v_j},\sigma_j^- \,\reg}_{0,3,s_2}(\iota_* [\cM^{E_{v_i}}_{0,3,s_1}(X_0,X_z,\cD^{E_{v_i}}_\infty)]_\virt, X_{z'},\pt)\right]_\virt.
\end{align*}
We take $\cF(\pt)^\reg\subset \cF(\pt)$ to be union of those connected components whose fibers at zero are components of $\cF_0(\pt)^{\sigma_j^- \,\reg}$.  A generic fiber $\cM_{0,4,\sigma_i + \sigma_j^- + \alpha}^{E_{v_i-v_j},\reg}(X_0,X_z,X_{z'},\pt)$ is a union of those components of $\cM_{0,4,\sigma_i + \sigma_j^- + \alpha}^{E_{v_i-v_j}}(X_0,X_z,X_{z'},\pt)$ which contain a rational curve with one sphere component passing through $\pt$ representing a section class $s$.  This restricted degeneration family gives
$$\langle X_0, X_z,X_{z'},\pt \rangle_{0,4,\sigma_i + \sigma_j^- + \alpha}^{E_{v_i-v_j},\reg} = \sum_{s_1 + s_2=\sigma_i + \sigma_j^- + \alpha} \langle \iota_* [\cM^{E_{v_i}}_{0,3,s_1}(X_0,X_z,\cD^{E_{v_i}}_\infty)]_\virt,X_{z'},\pt \rangle^{E_{-v_j},\sigma_j^- \,\reg}_{0,3,s_2},$$
where the left-hand side is by definition the integration of $1$ over the virtual fundamental class associated to $\cM_{0,4,\sigma_i + \sigma_j^- + \alpha}^{E_{v_i-v_j},\reg}(X_0, X_z,X_{z'},\pt)$.

Since the Seidel element $S_i = \sum_{s_1} q^{s_1} [\cM^{E_{v_i}}_{0,1,s_1}(\cD^{E_{v_i}}_\infty)]_\virt = \sum_{s_1} q^{s_1} [\cM^{E_{v_i}}_{0,3,s_1}(X_0,X_z,\cD^{E_{v_i}}_\infty)]_\virt$ is a divisor in $X$ (where the last equality is by divisor equation since $X_0, X_z$ are divisors in $E_{v_i}$ and $s_1 \cdot X_0 = s_1 \cdot X_z = 1$), only $s_1$ with $c_1(s_1) = 1$ contributes.  But $E_{v_i}$ is semi-Fano and so any section class (which is $\sigma_i + \alpha$ for some $\alpha$) has $c_1 \geq 1$.  Thus $s_1 = \sigma_i + \alpha_1$ where $\alpha_1<\alpha$ has $c_1(\alpha_1)=0$.  Moreover $c_1(\alpha)=0$ by dimension counting on the left-hand side.  Then $s_2 = \sigma_j^-+\alpha_2$ for some $\alpha_2$ satisfying $c_1(\alpha_2)=0$ and $\alpha_1+\alpha_2=\alpha$.

Now summing over $\alpha$ gives
\begin{equation} \label{eq:sum}
\sum_\alpha q^{\alpha} \langle X_0,X_z,X_{z'},\pt \rangle_{0,4,\sigma_i + \sigma_j^- + \alpha}^{E_{v_i-v_j},\reg} = \sum_{\alpha_2,l} q^{\alpha_2} \langle \iota_* S_i^\circ,X_{z'},\pt \rangle^{E_{-v_j},\sigma_j^- \,\reg}_{0,3,\sigma_j^- + \alpha_2}.
\end{equation}
By Lemma \ref{lem:reg=>unique} and its proof, every rational curve in $\cM^{E_{-v_j},\sigma_j^- \,\reg}_{0,3,\sigma_j^- + \alpha_2}(\iota_* D_l,X_{z'},\pt)$ is a union of a holomorphic sphere representing $\sigma_j^-$ and a rational curve supported in $\cD_0$ representing $\alpha$.  Such a rational curve intersects $X_{z'}$ at exactly one point (we take $z'$ such that $X_{z'} \not= \cD_0 \subset E_{-v_j}$), and hence $\cM^{E_{-v_j},\sigma_j^- \,\reg}_{0,3,\sigma_j^- + \alpha_2}(\iota_* D_l,X_{z'},\pt) \cong \cM^{E_{-v_j},\sigma_j^- \,\reg}_{0,2,\sigma_j^- + \alpha_2}(\iota_* D_l,\pt)$.  So the right-hand side of \eqref{eq:sum} is exactly the quantity we want to compute, namely,
$$\sum_\alpha q^{\alpha} \langle \iota_* S_i^\circ, \pt \rangle^{E_{-v_j},\sigma_j^- \, \reg}_{0,2,\sigma_j^-+\alpha}.$$

Now consider the left-hand side of \eqref{eq:sum}.  The moduli space contains a rational curve with a sphere component passing through $\pt$  representing a section class $s$.  Moreover, by dimension counting, the invariant is non-zero only when $c_1(\sigma_i + \sigma_j^- + \alpha) = 2$.  Since $X$ is semi-Fano, the sphere component representing $s$ which does not lie in any toric divisor and intersect each toric divisor transversely has $c_1 \leq 2$.  On the other hand, Since $s$ is a section class, it intersects $\cD_0$ and $\cD_\infty$ once.  Suppose $i\neq j$.  In order to have the balancing condition $\sum_{i\in\{0,1\ldots,m,\infty\}} \left(\cD_i \cdot s\right) v^E_i = 0$, the sphere component must intersect some divisors other than $\cD_0$ and $\cD_\infty$ (because $v^E_0 + v^E_\infty = (v_i-v_j,0) \not=0$).  This implies $s$ has $c_1 > 2$, a contradiction.  Thus the left-hand side of \eqref{eq:sum} is simply zero when $i \neq j$.  When $i=j$, $E_{v_i-v_j}$ is the trivial bundle $X \times \bP^1$, and $\sigma_i + \sigma_j^-$ is the constant section $\bP^1 \to X \times \bP^1$.  Thus the invariant $\langle X_0,X_z,X_{z'},\pt \rangle_{0,4,\sigma_i + \sigma_j^- + \alpha}^{E_{v_i-v_j},\reg}$ is one when $\alpha=0$, and zero otherwise.  Hence the left-hand side is $\delta_{ij}$.  This proves \eqref{eqn:2pt_reg_inv}.
\end{proof}

We are now ready to prove our main theorem:
\begin{thm}[=Theorem \ref{thm delta_intro}] \label{Thm g-open}
For all $j=1,\ldots,m$,
$$ \exp (g_j(\check{q}(q))) = \sum_{\alpha \in H_2^{c_1=0}(X)} q^{\alpha} n_1(\beta_j + \alpha) = 1 + \delta_{j}(q). $$
\end{thm}
\begin{proof}
The idea is to use Theorem \ref{Thm open-closed} to identify open GW invariants of $X$ with some closed GW invariants of the Seidel spaces, and then use \eqref{eq:main} to compute these closed invariants.

For the left-hand side of the formula we want to deduce, by Corollary \ref{cor:g}, $\exp(g_j(\check{q}(q)))$ only involves Novikov variables $q^\alpha$ with $\alpha \in H_2^{\eff,c_1=0}(X)$ satisfying $D_l \cdot \alpha = 0$ for $v_l\notin F(v_j)$.  For the right-hand side, by Corollary \ref{cor:n}, $\sum_\alpha q^\alpha n_1(\beta_j+\alpha)$ also has only Novikov variables $q^\alpha$ with $D_i \cdot \alpha = 0$ whenever $v_i \not\in F(v_j)$.  Thus if $v_i \not\in F(v_j)$, then
$$\hat{D}_i \cdot \exp(g_j(\check{q}(q))) = \hat{D}_i \cdot \left(\sum_\alpha q^\alpha n_1(\beta_j+\alpha)\right) = 0.$$
In the following we prove that the above equality also holds in the case when $v_i \in F(v_j)$.

Taking $\sum_\alpha q^\alpha \cdot$ on both sides of Equation \eqref{eq:main} and applying Proposition \ref{prop:delta} to the right-hand side, we have
$$
\sum_\alpha q^\alpha \langle \iota_* D_i, [\pt] \rangle^{E_j^-,\,\sigma_j^-\,\reg}_{0,2,\sigma_j^-+\alpha}\\
=\delta_{ij} \exp (g_i(\check{q}(q))) + (\hat{D}_i \cdot g_j(\check{q}(q))) \exp (g_j(\check{q}(q)))
$$
Combining with Theorem \ref{Thm open-closed}, we have
$$ \exp (g_j(\check{q}(q))) (\delta_{ij} + \hat{D}_i \cdot g_j(\check{q}(q))) = \sum_{\alpha \in H_2^{c_1=0}(X)} q^{\alpha} n_{1,1}(\beta_j + \alpha; D_{i}, [\pt]_X). $$
Thus
\begin{align*}
\hat{D}_{i} \cdot \big(Q_j \exp (g_j(\check{q}(q))) \big) &= \sum_{\alpha \in H_2^{c_1=0}(X)} Q_j q^{\alpha} n_{1,1}(\beta_j + \alpha; D_{i}, [\pt]_X) \\
&= \hat{D}_{i} \cdot\left( \sum_{\alpha \in H_2^{c_1=0}(X)} Q_j q^{\alpha} n_1(\beta_j + \alpha) \right)
\end{align*}
where we recall that $Q_j$ is a coordinate on the extended K\"ahler moduli $\tilde{\Kcone}^\C_X$ and $\hat{D}_{i} \cdot Q_j = \delta_{ij}$ (Section \ref{sec:extmod}), and the last equality follows from Theorem \ref{thm div eq open} (the divisor equation). This proves that the above equality holds for all $D_i$, and the theorem follows.
\end{proof}

\subsection{Corollaries}
We now describe some consequences of Theorem \ref{Thm g-open}.


\begin{thm} \label{thm conv}
The coefficients of the disc potential $W^{\mathrm{LF}}$ of a compact semi-Fano toric manifold $X$ are convergent power series in the K\"ahler parameters $q^\nef$.
\end{thm}

\begin{proof}
This follows from Theorem \ref{Thm g-open} and the fact that the hypergeometric series $g_j(\check{q}^\nef)$ and the inverse mirror map $\check{q}^\nef(q^\nef)$ are convergent.
\end{proof}

\begin{cor} \label{cor inv mir map}
The inverse mirror map $\check{q}(q)$ of a compact semi-Fano toric manifold $X$ is written in terms of the generating functions $\delta_l$ of open GW invariants as
$$\check{q}_k(q) = q_k \prod_{l=1}^m (1+\delta_l(q))^{D_l \cdot \Psi_k} = q_k (1+\delta_{n+k}(q)) \prod_{p=1}^n (1+\delta_p(q))^{-\pairing{v_{n+k}}{\nu_p}}.$$
\end{cor}
\begin{proof}
By \eqref{mirror_map_eqn}, we have
$$\check{q}_k (q) = q_k \exp (g^{\Psi_k}(\check{q}(q))) = q_k \prod_{l=1}^m (\exp g_l (\check{q}(q)))^{D_l \cdot \Psi_k}$$
and thus the equality follows from Theorem \ref{Thm g-open}.
Also, $\Psi_k = \beta_{n+k} - \sum_{p=1}^{n} \pairing{v_{n+k}}{\nu_p} \beta_p$, and so $D_p \cdot \Psi_k = -\pairing{v_{n+k}}{\nu_p}$ for $p=1,\ldots,n$ and $D_{n+r} \cdot \Psi_k = \delta_{rk}$ for $r=1,\ldots,m-n$.
\end{proof}

\begin{proof}[Proof of Theorem \ref{thm W equal}]
Recall that the Hori-Vafa potential $\tilde{W}^{\mathrm{HV}}_{\check{q}}$ is written as \eqref{eqn W tilde}:
\begin{equation*}
\tilde{W}^{\mathrm{HV}}_{\check{q}}(z_1, \ldots, z_n) = \sum_{p=1}^n (\exp g_p(\check{q})) z_p +
\sum_{k=1}^{m-n} \check{q}_{k} z^{v_{n+k}} \prod_{p=1}^n \exp \left(g_p(\check{q})\right)^{\pairing{v_{n+k}}{\nu_p}}.
\end{equation*}
Then by Theorem \ref{Thm g-open} and Corollary \ref{cor inv mir map}, we have
\begin{align*}
\tilde{W}^{\mathrm{HV}}_{\check{q}(q)} = \sum_{p=1}^n (1 + \delta_p(q)) z_p + \sum_{k=1}^{m-n} q_k z^{v_{n+k}} (1+\delta_{n+k}(q))
= W^{\mathrm{LF}}_{q}.
\end{align*}
\end{proof}


\begin{proof}[Proof of Theorem \ref{thm Seidel to gen intro}]
By Theorem \ref{thm W equal}, since  $W^{\mathrm{LF}}$ and $\tilde{W}^{\mathrm{HV}}$ are equal, $\QH^*(X, \omega_q)\overset{\simeq}\to \Jac(W_q^{\mathrm{LF}})$ is the same as $\QH^*(X, \omega_q)\overset{\simeq}\to \Jac(\tilde{W}_{\check{q}(q)}^{\mathrm{HV}})$.  By Proposition \ref{prop Bat to}, each Batyrev element $B_l$ is mapped to $(\exp g_l) Z_l$ for $l=1,\ldots,m$.  Since $B_l = (\exp g_l) S^\circ_l$ by Proposition \ref{prop Seidel Batyrev}, it follows that $S^\circ_l$ is mapped to $Z_l$ for $l=1,\ldots,m$.
\end{proof}

We conjecture that Theorem \ref{thm Seidel to gen intro} holds true for {\em any} compact toric manifold:
\begin{cnjc}\label{conj:seidel_elt}
Let $X$ be a compact toric manifold, not necessarily semi-Fano. Then the isomorphism \eqref{FOOO_isom_intro} maps the normalized Seidel elements $S^{\circ}_l\in \QH^*(X, \omega_q)$ to the generators\footnote{When $X$ is not even semi-Fano, $W_q^{\mathrm{LF}}$ is in general a Laurent series, instead of a Laurent polynomial, over the Novikov ring. Nevertheless we can still define the monomials $Z_l$ by Equation \eqref{Eqn Z}.} $Z_l$ of the Jacobian ring $\Jac(W_q^{\mathrm{LF}})$, where $Z_l$ are monomials defined by Equation \eqref{Eqn Z}.
\end{cnjc}

\begin{ex}
Consider the semi-Fano toric surface $X$ whose moment map image is shown in Figure \ref{fig egsftor}.  The disc potential $W_q^{\mathrm{LF}}$ and generating functions $\delta_i(q)$ of $X$ were computed in \cite{chan-lau}.  The key result is that $n_1(\beta) = 1$ when $\beta$ is an admissible disc class, and $n_1(\beta) = 0$ otherwise.  Admissibility is a combinatorial condition which is easy to check, and the readers are referred to \cite{chan-lau} for the detailed definitions and results.

\begin{figure}[htp]
\begin{center}
\includegraphics[scale=0.8]{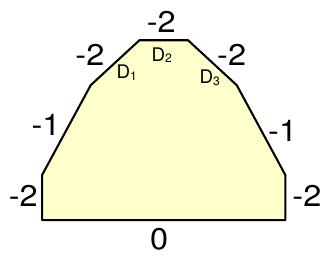}
\end{center}
\caption{An example of semi-Fano toric surface ($X_8$ in \cite[Appendix A]{chan-lau}).}
\label{fig egsftor}
\end{figure}

The generating functions corresponding to $D_1,D_2,D_3$ are
\begin{align*}
\delta_1(q) &= q_1 + q_1 q_2 + q_1 q_2 q_3,\\
\delta_2(q) &= q_2 + q_1 q_2 + q_2 q_3 + q_1 q_2 q_3 + q_1 q_2^2 q_3,\\
\delta_3(q) &= q_3 + q_2 q_3 + q_1 q_2 q_3
\end{align*}
respectively, where $q_i$'s are the K\"ahler parameters of $D_i$'s for $i=1,2,3$.  Each term in the above generation functions corresponds to an admissible disc class.

On the other hand, the mirror map is given by
\begin{align*}
q_1 &= \check{q}_1 \exp \big(2 g_1(\check{q}_1,\check{q}_2,\check{q}_3) - g_2(\check{q}_1,\check{q}_2,\check{q}_3)\big);\\
q_2 &= \check{q}_2 \exp \big(-g_1(\check{q}_1,\check{q}_2,\check{q}_3) + 2 g_2(\check{q}_1,\check{q}_2,\check{q}_3) - g_3(\check{q}_1,\check{q}_2,\check{q}_3)\big);\\
q_3 &= \check{q}_3 \exp \big(-g_2(\check{q}_1,\check{q}_2,\check{q}_3) +2 g_3(\check{q}_1,\check{q}_2,\check{q}_3)\big)
\end{align*}
where
\begin{align*}
g_1(\check{q}_1,\check{q}_2,\check{q}_3) &= \sum_{(a,b,c)} \check{q}_1^a\check{q}_2^b\check{q}_3^c \frac{(-1)^{2a-b} (2a-b-1)!}{a!c!(a-2b+c)!(b-2c)!};\\
g_2(\check{q}_1,\check{q}_2,\check{q}_3) &= \sum_{(a,b,c)} \check{q}_1^a\check{q}_2^b\check{q}_3^c \frac{(-1)^{2b-a-c} (2b-a-c-1)!}{a!c!(b-2a)!(b-2c)!};\\
g_3(\check{q}_1,\check{q}_2,\check{q}_3) &= \sum_{(a,b,c)} \check{q}_1^a\check{q}_2^b\check{q}_3^c \frac{(-1)^{2c-b} (2c-b-1)!}{a!c!(a-2b+c)!(b-2a)!}
\end{align*}
where the summations are over all $(a,b,c)\in\integer^3$ such that the term before each factorial sign is non-negative.
By Theorem \ref{Thm g-open}, we have $1+\delta_i(q(\check{q})) = \exp g_i(\check{q})$ for $i=1,2,3$. This produces non-trivial identities between hypergeometric series, and hence a closed formula for the inverse mirror map $\check{q}(q)$:

\begin{align*}
\check{q}_1 &= q_1 \cdot \frac{1+q_2 + q_1 q_2 + q_2 q_3 + q_1 q_2 q_3 + q_1 q_2^2 q_3}{(1+q_1 + q_1 q_2 + q_1 q_2 q_3)^{2}};\\
\check{q}_2 &= q_2 \cdot \frac{(1+q_1 + q_1 q_2 + q_1 q_2 q_3)(1+q_3 + q_2 q_3 + q_1 q_2 q_3)}{(1+q_2 + q_1 q_2 + q_2 q_3 + q_1 q_2 q_3 + q_1 q_2^2 q_3)^{2}};\\
\check{q}_3 &= q_3 \cdot \frac{1+q_2 + q_1 q_2 + q_2 q_3 + q_1 q_2 q_3 + q_1 q_2^2 q_3}{(1+q_3 + q_2 q_3 + q_1 q_2 q_3)^{2}}.
\end{align*}

\end{ex}

\subsection{Equivalence of results}
In fact the statements in Corollary \ref{cor inv mir map} and Theorems \ref{thm W equal}, \ref{Thm g-open}, and \ref{thm Seidel to gen intro} are all equivalent to each other.
\begin{prop} \label{prop:equiv}
Let $X$ be a semi-Fano toric manifold.  The following statements are equivalent.
\begin{enumerate}
\item The inverse mirror map is equal to
$$\check{q}_k(q) = q_k \prod_{l=1}^m (1+\delta_l(q))^{D_l \cdot \Psi_k} = q_k (1+\delta_{n+k}(q)) \prod_{p=1}^n (1+\delta_p(q))^{-\pairing{v_{n+k}}{\nu_p}}.$$
\item The generating function of open Gromov-Witten invariants is given by
$$ \sum_{\alpha \in H_2^{c_1=0}(X)} q^{\alpha} n_1(\beta_j + \alpha) = 1+\delta_j(q) = \exp (g_j(\check{q}(q))).$$
\item The disc potential is equal to the Hori-Vafa superpotential via the inverse mirror map:
$$ W^{\mathrm{LF}}_{q}=\tilde{W}^{\mathrm{HV}}_{\check{q}(q)}.$$
\item The (normalized) Seidel elements $S^\circ_l\in \QH^*(X, \omega_q)$ are mapped to the generators $Z_1, \ldots, Z_m$ of the Jacobian ring $\Jac(W_q^{\mathrm{LF}})$ (see Equation \eqref{Eqn Z}) under the isomorphism \eqref{FOOO_isom_intro}.
\end{enumerate}
\end{prop}

We have seen that (2) implies (1) which then implies (3). Conversely,
suppose that we have (3). Then (1) holds by definition of the potentials. On the other hand, McDuff-Tolman \cite[Proposition 5.2]{McDuff-Tolman} show that the normalized Seidel elements $S^\circ_j(q)$ satisfy the multiplicative relations:
$$\prod_{l=1}^m S^\circ_j(q)^{D_l \cdot d} = q^d$$
for any $d\in H_2(X,\integer)$.
Together with the multiplicative relations \eqref{eqn Batyrev multi} satisfied by the Batyrev elements and Proposition \ref{prop Seidel Batyrev}, we obtain
\begin{equation}\label{eqn multi rel}
\prod_{l=1}^m (1+\delta_l(q))^{D_l \cdot d} = \prod_{l=1}^m \exp \left(g_l(\check{q}(q))\right)^{D_l \cdot d}
\end{equation}
for any $d\in H_2(X,\integer)$. To see that this implies (2), we need the following\footnote{This lemma is obviously a consequence of (2), but here we need to prove it without assuming (2).}

\begin{lem}
If $g_l(\check{q})$ vanishes, then so does $\delta_l(q)$.
\end{lem}
\begin{proof}
Suppose that $\delta_l \neq 0$.  Then there exists $\alpha \in H_2(X)$ represented by a rational curve with Chern number zero such that $n_{\beta_l + \alpha} \neq 0$.  The class $\alpha$ is represented by a tree $C$ of rational curves in $X$. Let $C'$ be the irreducible component of $C$ which intersects with the disk representing $\beta_l$. Let $d=[C']\in H_2(X)$. Then the Chern number of $d$ is also zero since $X$ is semi-Fano. Furthermore, $D_l\cdot d<0$ because the invariance of $n_{\beta_l+\alpha}$ under deformation of the Lagrangian torus fiber $L$ implies that $C'$ is contained inside the toric divisor $D_l$. We claim that $D_j\cdot d\geq0$ for all $j\neq l$. When $n=2$, this is obvious. When $n\geq3$, $D_j\cdot d<0$ for some other $j\neq l$ implies that the curve $C'$ is contained inside the codimension two subvariety $D_l\cap D_j$. However, the intersection of $C'$ with the disk representing $\beta_l$ cannot be inside $D_l\cap D_j$ since $\beta_l$ has Maslov index 2. So we conclude that $D_j\cdot d\geq0$ for all $j\neq l$.
Thus $d = [C'] \in H_2(X)$ satisfies the properties that
\begin{align*}
-K_X \cdot d = 0, D_l \cdot d < 0\text{ and } D_j \cdot d \geq 0 \text{ for all } j\neq l,
\end{align*}
which contributes to a term of $g_l(\check{q})$, and hence $g_l(\check{q}) \neq 0$ (distinct $d$ leads to distinct $\check{q}^d$, and hence they do not cancel each other).
\end{proof}

Now consider $A_l(q):=\log\left(e^{-g_l(\check{q})}(1+\delta_l(q))\right)$, $l=1,\ldots,m$.
By \cite[Proposition 4.3]{G-I11}, $g_l$ vanishes if and only if $v_l$ is a vertex of the fan polytope of $X$, and
any convex polytope with nonempty interior in $\real^n$ has at least $n+1$ vertices, so at least $n+1$ of the functions $g_l$ are vanishing (cf. \cite[Corollary 4.6]{G-I11}). Thus the above lemma implies that at least $n+1$ of the functions $A_l$ are vanishing.
Without loss of generality, assume that $g_1,\ldots,g_s$ (with $s \leq m-n-1$) are the non-vanishing functions so that $A_l \equiv 0$ for $l>s$. Taking logarithms on both sides of \eqref{eqn multi rel} we have the following equality for any $d\in H_2(X,\integer)$:
\begin{equation}\label{eqn log multi rel}
\sum_{l=1}^s \left( D_l \cdot d \right) A_l(q)=0.
\end{equation}

For $l=1,\ldots,s$ (when $v_l$ is not a vertex of the fan polytope), recall that $F(v_l)$ is the minimal face of the fan polytope of $X$ containing $v_l$. Then $F(v_l)$ is the convex hull of primitive generators $v_{p_1},\ldots,v_{p_k}$ which are vertices of the fan polytope of $X$. So there exist integers $a_1,\ldots,a_k,b_l>0$ such that $a_1v_{p_1}+\ldots+a_kv_{p_k}-b_lv_l=0$. This primitive relation corresponds to a class $d_l\in H_2(X,\integer)$ such that $D_l\cdot d_l = -b_l < 0$, $D_{p_t}\cdot d_l = a_t$ and $D_r\cdot d_l = 0$ when $r$ is none of $l,p_1,\ldots,p_k$. (cf. proof of \cite[Theorem 1.2]{G-I11}.)  Then the Equation \eqref{eqn log multi rel} for the class $d = d_l$ is simply given by
$$-b_l A_l = 0,$$
whence $A_l \equiv 0$ for $l = 1,\ldots,s$. This proves (2).

We have seen that (2) implies (3), which in turn implies (4).  Now suppose (4) holds, i.e.  the isomorphism \eqref{FOOO_isom_intro} maps $S^\circ_l$ to $Z_l$ for $l=1,\ldots,m$, then the elements $\tilde{B}_l\in \QH^*(X, \omega_q)$ defined by $\tilde{B}_l:=(1+\delta_l)S^\circ_l$ satisfy the conditions (i), (ii) and (iii) of \cite[Theorem 1.2]{G-I11}, which states that these conditions completely characterize the Batyrev elements, so that we have $\tilde{B}_l=B_l$ in $\QH^*(X, \omega_q)$. (2) then follows from Proposition \ref{prop Seidel Batyrev}.  This completes the proof of Proposition \ref{prop:equiv}.

\bibliographystyle{amsplain}
\bibliography{geometry}

\end{document}